\documentclass [a4paper,10pt]{article}
\usepackage{amsmath}
\usepackage{amsfonts}
\usepackage{amssymb}
\usepackage{booktabs}
\usepackage{multirow}
\usepackage{array,longtable}
\usepackage{amscd}
\usepackage[cp1251]{inputenc}
\usepackage[english]{babel}
\usepackage[usenames]{color}
\usepackage[matrix,arrow,curve]{xy}
\usepackage{tikz}
\usepackage{pgfplots}
\pgfplotsset{compat=1.18}
\usepgfplotslibrary{fillbetween}
\usetikzlibrary{shapes.geometric, arrows.meta, positioning, calc}
\usepackage[margin=0.6in]{geometry}
\usetikzlibrary{positioning, calc}

\begin{document}
\newtheorem{definition}{Definition}
\newtheorem{theorem}{Theorem}
\newtheorem{lemma}{Lemma}
\newtheorem{proposition}{Proposition}
\newtheorem{conjecture}{Conjecture}
\newtheorem{criterion}{Criterion}
\newtheorem{observation}{Observation}
\newtheorem{corollary}{Corollary}
\newcommand{\diag}{{\rm diag}}
\newcommand{\tr}{{\rm Tr}}
\begin{center}
{\large {\bf A RECURSIVE DETERMINANTAL FRAMEWORK \\ FOR TESTING $D$-STABILITY.}} \\[0.5cm]
{\bf Olga Y. Kushel} \\[0.5cm]
Institute of Mathematics of NAS of Belarus, \\ Surganova St. 11, \\  220072 Minsk, Republic of Belarus \\
kushel@mail.ru
\end{center}

\begin{abstract}
The concept of matrix $D$-stability, introduced in 1958 by Arrow and McManus, is of major importance across a wide variety of applications in economic modeling, ecology, and control systems. However, an exact algebraic characterization of $D$-stability for dimensions $n > 4$ has remained a notoriously intractable open problem for over sixty years. In this paper, we establish a novel, systematic recursive framework that decomposes the structural check of $D$-stability into an analytical tree of parameter-dependent determinants. By applying a recursive delete/zero reduction strategy, we derive exact recurrence relations for the real and imaginary parts of the characteristic polynomial components. These algebraic relations uncover a structured hierarchy of new sufficient conditions for $D$-stability, expressed explicitly in terms of the matrix's principal minors. We show that while general numerical methods face unavoidable conservatism near the topological boundaries of the stable manifold, our deterministic framework provides sharp, absolute certification for low-order boundary matrices. 
\end{abstract}



 \section*{Introduction}
Let $\mathcal{M}^{n\times n}$ be
the set of all square real\ $n\times n$ matrices; $\sigma (\mathbf{A})$ be the
spectrum of a matrix $\mathbf{A}$ $\in \mathcal{M}^{n\times n}$ (i.e. the set of all eigenvalues of $\mathbf A$ defined as zeroes of its characteristic polynomial $f_{\mathbf A}(\lambda):= \det({\mathbf A} - \lambda{\mathbf I})$). Recall, that a matrix $\mathbf{A}$ $\in \mathcal{M}^{n\times n}$ is called {\it (positive) stable} if ${\rm Re}(\lambda )>0$ for all $\lambda \in \sigma (\mathbf{A})$ (see \cite{CA}, \cite{HS} and many others). 
 
The concept of matrix $D$-stability, first introduced by Arrow and McManus in 1958 (\cite{AM}), has proven to be of lasting importance due to its wide-ranging applications in economics, ecology, and control theory (\cite{KAB}, \cite{KU2}). A matrix $\mathbf{A} \in \mathcal{M}^{n\times n}$ is
called {\it $D$-stable} if ${\rm Re}(\lambda )>0$ for all $\lambda \in
\sigma (\mathbf{DA})$, where ${\mathbf D}$ is an arbitrary positive diagonal matrix. This property ensures, for instance, that the equilibrium of a multi-market economy remains stable regardless of individual market adjustment speeds (see, for example, \cite{GZ}, \cite{QR}).

Despite over six decades of intensive research, a complete characterization of $D$-stable matrices remains a formidable open problem for dimensions $n > 4$. The primary difficulty stems from the fact that the condition of $D$-stability involves $n$ independent, continuously varying parameters across an unbounded domain. While necessary conditions, such as the requirement that a positive $D$-stable matrix be a $P^+_0$-matrix \cite{QR}, are well-understood, and sufficient conditions exist for special subclasses \cite{JOHN1}, a verifiable criterion applicable to matrices of arbitrary size has remained elusive. Elegant existential formulations, most notably Johnson's criterion \cite{JOHN5}, demonstrate that $D$-stability is equivalent to the non-vanishing of $\det({\mathbf A} + i{\mathbf D})$ for all ${\mathbf D} = \{d_1, \ \ldots, \ d_n\}$, $d_i > 0$. Johnson formalized this via the simultaneous non-solvability of the real and imaginary parts ($P_{\mathbf A}^2(d_{1}, \dots, d_{n})$ and $Q_{\mathbf A}^2(d_{1}, \dots, d_{n})$, respectively) of the characteristic determinant, mapping the property to a single polynomial equation: $$F(d_1, \ \ldots, \ d_n): = P_{\mathbf A}^2(d_{1}, \dots, d_{n}) + Q_{\mathbf A}^2(d_{1}, \dots, d_{n}) = 0.$$ 

Despite its elegant formulation, Johnson's criterion remains fundamentally non-algorithmic for dimensions $n>2$, since it deals with a polynomial in $n$ unbounded variables. It provides neither explicit algebraic expressions for the coefficients in terms of the matrix's principal minors, nor an operational procedure to verify the absence of positive roots. Standard testing tools, such as Descartes' rule of signs, cannot be readily applied to $P_{\mathbf A}^2 + Q_{\mathbf A}^2$ without computing an intractable number of implicit polynomial coefficients. Semidefinite programming relaxations and sums-of-squares techniques, while powerful for fixed-size instances, do not scale gracefully here because the domain ${\mathbb R}^n_{+}$ is non-compact and the degrees of the arising polynomials grow with each relaxation order.

To bridge this gap, we introduce a alternative polynomial framework $F(0,1)(d_{1}, \dots, d_{n-1})$. Crucially, the coefficients of $F(0,1)$ do not coincide with those of Johnson's $P_{\mathbf A}^2 + Q_{\mathbf A}^2$. By exploiting the structural properties of $F(0,1)$, we show that Descartes' rule of signs can be operationalized into a finite, verifiable hierarchy of inequalities governed strictly by the principal minors of the matrix. 

The main goal of this paper is to propose a new, constructive approach to testing matrix $D$-stability. Basing our work on Johnson's $D$-stability criterion \cite{JOHN5}, we develop a recursive "delete/zero" algorithm to expand $\det({\mathbf A} + i{\mathbf D})$ in terms of the diagonal entries of $\mathbf D$. This expansion recursively decomposes the $D$-stability criterion into relations involving the determinants of smaller matrices. Thus we obtain a hierarchical set of conditions expressed solely as inequalities between the principal minors of the original matrix $\mathbf A$. A key feature of our framework is that it sits between the two extremes of "intractably exact" and "overly conservative". Crucially, the algorithm can be truncated at any depth, producing valid sufficient conditions for $D$-stability whose complexity and conservatism are tunable. For low-order matrices and certain structured families, the deepest level of the hierarchy provides sharp certification beyond existing LMI tests (see Section 7.4).

 As a theoretical culmination of our recursive process, we present a set of sufficient conditions for $D$-stability in the form of explicit inequalities involving the principal minors of an 
$n \times n$ matrix $\mathbf A$. These inequalities define a class of matrices not captured by previously known sufficient criteria. 

The paper is organized as follows. Section 1 introduces notations, basic results on determinant expansions and proves the determinantal identities. Section 3 presents the Delete/Zero algorithm and derives the core recurrence relations.  Section 3 reviews fundamental properties of D-stable matrices and recalls Johnson's criterion. Section 4 analyzes the first recursion step, culminating in Theorem \ref{Step1}. Section 5 gives detailed analysis of the second step, including degenerate case. Two options of hierarchical $D$-stability tests are provided in Section 6. We demonstrate the advantages of our approach by numerical experiments in Section 7.

 \section{Core matrix identities}
 \subsection{Preliminary definitions and notations}
 Here, as usual, we use the notation $[n]$ for the set of indices $\{1,2, \ \ldots, \ n\}$. Given a set $\alpha = (i_1, \ \ldots, i_k)$, with $1 \leq i_1 < \ldots < i_k \leq n$, we denote $N(\alpha)$ the cardinality of $\alpha$, i.e. the number of elements in $\alpha$ (obviously, $N(\emptyset) = 0$). 
 
 Given an $n \times n$ complex matrix ${\mathbf A} = \{a_{ij}\}_{i,j = 1}^n$, we introduce the following notations:
\begin{enumerate} 
\item[-] $A \begin{pmatrix}\alpha \\ \alpha \end{pmatrix}$ (or $A\begin{pmatrix}i_1 & \ldots & i_k \\ i_1 & \ldots & i_k \end{pmatrix}$) for the principal minor of $\mathbf A$, which lies on the intersection of the rows and columns with indices from $\alpha = (i_1, \ \ldots, i_k)$.
\item[-] ${\mathbf A}|_{n}$ for the $(n-1) \times (n-1)$ principal submatrix of $\mathbf A$, obtained by deleting the $n$th row and the $n$th column.
\end{enumerate} 
Then ${\mathbf A} = \{a_{ij}\}_{i,j = 1}^n$ can be written in the following block form
\begin{equation}\label{ym2}
{\mathbf A} = \begin{pmatrix} {\mathbf A}|_{n} & \overline{{\mathbf a}}_{1n} \\
\underline{{\mathbf a}}_{n1} & a_{nn} \end{pmatrix},
\end{equation}
where $\overline{{\mathbf a}}_{1n} = (a_{1n}, \ \ldots, \ a_{n-1, n})^T$, $\underline{{\mathbf a}}_{n1} = (a_{n1}, \ \ldots, \ a_{n, n-1})$.

Assuming that $a_{nn} \neq 0$, we use the notation ${\mathbf A}/{a_{nn}}$ for the {\it Schur complement} of $a_{nn}$ in $\mathbf A$, defined as follows:
\begin{equation}\label{sh1}{\mathbf A}/{a_{nn}}  = {\mathbf A}|_n - \frac{1}{a_{nn}}\overline{{\mathbf a}}_{1n}\underline{{\mathbf a}}_{n1}.\end{equation}

Now consider a diagonal matrix ${\mathbf D} \in {\mathcal M}^{n \times n}$:
$${\mathbf D} = \diag\{d_1, \ \ldots, \ d_{n-1}, \ d_n\}, \qquad d_i \in {\mathbb R}.$$
In this case, ${\mathbf D}|_n = \diag\{d_1, \ \ldots, \ d_{n-1}\}$ (while $ {\mathbf D} = \diag\{{\mathbf D}|_n, \ d_n\}$)
and we shall use the notation ${\mathbf D}|_n^0$, for the diagonal matrix defined by setting $d_n:=0$:
$${\mathbf D}|_n^0 = \diag\{d_{1}, \ \ldots, \ d_{n-1}, \ 0\}, \qquad d_i \in {\mathbb R}. $$

\subsection{Determinant expansions}
Later, we shall use the following matrix determinant expansion (see, for example, \cite{BER}, p. 133, Fact 2.14.2).
\begin{lemma}\label{1}
Let an $n \times n$ complex matrix ${\mathbf A} = \{a_{ij}\}_{i,j = 1}^n$ be represented in Form \eqref{ym2}. Then
\begin{enumerate}
    \item[\rm 1.] $\det{\mathbf A} = \det({\mathbf A}|_{n})(a_{nn} - \underline{{\mathbf a}}_{n1}({\mathbf A}|_{n})^{-1}\overline{{\mathbf a}}_{1n})$ if $\det({\mathbf A}|_{n}) \neq 0$;
    \item[\rm 2.] $\det{\mathbf A} = a_{nn}\det({\mathbf A}|_{n} - \frac{1}{a_{nn}}\overline{{\mathbf a}}_{n1}\underline{{\mathbf a}}_{1n})$ if $a_{nn} \neq 0$;
    \item[\rm 3.] $\det{\mathbf A} = - \underline{{\mathbf a}}_{n1}{\rm adj}({\mathbf A}|_{n})\overline{{\mathbf a}}_{1n}$ if  $a_{nn} = 0$ or $\det({\mathbf A}|_{n}) = 0$.
\end{enumerate}
\end{lemma}
Following \cite{CA}, we say that a matrix $\mathbf A$ is
\begin{enumerate}
\item[-] {\it of type I} (with respect to index $n$), if in the pair $(a_{nn}, \det({\mathbf A}|_{n}))$ one element vanishes without the other vanishes also;
\item[-] {\it of type II}, otherwise.
\end{enumerate}
Obviously, a matrix $\mathbf A$ with $a_{nn} \det({\mathbf A}|_{n}) \neq 0$ as well as with $a_{nn} = \det({\mathbf A}|_{n}) = 0$ is of type II.

\subsection{Branched identities}
\begin{lemma} Given an arbitrary matrix ${\mathbf A} \in {\mathcal M}^{n \times n}$ and a diagonal matrix ${\mathbf D} \in {\mathcal M}^{n \times n}$.
Then the following identities hold:
\begin{equation}\label{id1}
    \det({\mathbf A} \pm i{\mathbf D}|_n^0) = a_{nn}\det({\mathbf A}/{a_{nn}} \pm i {\mathbf D}|_n), \qquad \mbox{if} \qquad a_{nn} \neq 0;\end{equation}
\begin{equation}\label{id2}\det({\mathbf A} \pm i{\mathbf D}|_n^0) = - \underline{{\mathbf a}}_{n1}{\rm adj}({\mathbf A}|_{n}\pm i{\mathbf D}|_{n})\overline{{\mathbf a}}_{1n}, \qquad \mbox{if} \qquad a_{nn} = 0.\end{equation}
\end{lemma}
{\bf Proof.} Indeed, Representation \eqref{ym2} of the matrix ${\mathbf A} \pm i{\mathbf D}|_n^0$ gives
$${\mathbf A} \pm i{\mathbf D}|_n^0 = \begin{pmatrix} {\mathbf A}|_{n} & \overline{{\mathbf a}}_{1n} \\
\underline{{\mathbf a}}_{n1} & a_{nn} \end{pmatrix} \pm i\begin{pmatrix} {\mathbf D}|_{n} & \overline{{\mathbf 0}} \\
\underline{{\mathbf 0}} & 0 \end{pmatrix} = \begin{pmatrix} {\mathbf A}|_{n}\pm i{\mathbf D}|_{n} & \overline{{\mathbf a}}_{1n} \\
\underline{{\mathbf a}}_{n1} & a_{nn} \end{pmatrix}.$$
Let $a_{nn} \neq 0$. Applying Lemma \ref{1}, Part 2 to ${\mathbf A} \pm i{\mathbf D}|_n^0$, we obtain
$$ \det({\mathbf A} \pm i{\mathbf D}|_n^0) = a_{nn}\det\left({\mathbf A}|_n \pm i {\mathbf D}|_{n} - \frac{1}{a_{nn}}\overline{{\mathbf a}}_{1n}\underline{\mathbf {a}}_{n1}\right)$$

Rearranging the terms inside the determinant yields:
$$ \det({\mathbf A} \pm i{\mathbf D}|_n^0) = a_{nn}\det\left({\mathbf A}|_n  - \frac{1}{a_{nn}}\overline{{\mathbf a}}_{1n}\underline{\mathbf {a}}_{n1} \pm i {\mathbf D}|_{n}\right)$$

Finally, substituting the definition of the Schur complement ${\mathbf A}/{a_{nn}} = {\mathbf A}|_n - \frac{1}{a_{nn}}\overline{{\mathbf a}}_{1n}\underline{\mathbf {a}}_{n1}$ from Formula \eqref{sh1}, we arrive at:
$$ \det({\mathbf A} \pm i{\mathbf D}|_n^0) = a_{nn}\det({\mathbf A}/{a_{nn}} \pm i {\mathbf D}|_n)
$$ which gives exactly Identity \eqref{id1}.
 
 For the proof of Identity \eqref{id2}, it is enough apply Lemma \ref{1}, Part 3 to ${\mathbf A} \pm i{\mathbf D}|_n^0$ with $a_{nn} = 0$. $\square$

\begin{lemma}\label{explemma} Given an arbitrary matrix ${\mathbf A} \in {\mathcal M}^{n \times n}$ and a diagonal matrix ${\mathbf D} \in {\mathcal M}^{n \times n}$, the following identity hold:
\begin{equation}\label{expana!}\det({\mathbf A} \pm i{\mathbf D}) = \pm id_n\det({\mathbf A}|_{n}\pm i{\mathbf D}|_{n}) + \det({\mathbf A} \pm i {\mathbf D}|_n^0). \end{equation}
In particular,
\begin{enumerate}
    \item[\rm 1.] If $a_{nn} \neq 0$ and $\det({\mathbf A}|_{n} \pm i{\mathbf D}|_{n}) \neq 0$ then \begin{equation}\label{expan}\det({\mathbf A} \pm i{\mathbf D}) = \pm id_n\det({\mathbf A}|_{n}\pm i{\mathbf D}|_{n}) + a_{nn}\det({\mathbf A}/{a_{nn}} \pm i {\mathbf D}|_n). \end{equation}
    \item[\rm 2.] If $\det({\mathbf A}|_{n} \pm i{\mathbf D}|_{n}) \neq 0$ then \begin{equation}\label{expana}\det({\mathbf A} \pm i{\mathbf D}) = \pm id_n\det({\mathbf A}|_{n}\pm i{\mathbf D}|_{n}) + \det({\mathbf A} \pm i {\mathbf D}|_n^0). \end{equation}
    \item[\rm 3.] If $\det({\mathbf A}|_{n} \pm i{\mathbf D}|_{n}) = 0$, $a_{nn} = 0$ then \begin{equation}\label{expanb}\det({\mathbf A} \pm i{\mathbf D}) =  \det({\mathbf A} \pm i {\mathbf D}|_n^0). \end{equation}
\end{enumerate}
\end{lemma}
{\bf Proof.} First, we prove that for $a_{nn} \neq 0$, $\det({\mathbf A}|_{n} \pm i{\mathbf D}|_{n}) \neq 0$, the following equality is true:
\begin{equation}\label{new}\det({\mathbf A}|_{n}\pm i{\mathbf D}|_{n})(a_{nn} - \underline{{\mathbf a}}_{n1}({\mathbf A}|_n \pm i {\mathbf D}|_{n})^{-1}\overline{{\mathbf a}}_{1n})  = a_{nn}\det({\mathbf A}/{a_{nn}} \pm i {\mathbf D}|_n).\end{equation} Indeed, since $\det({\mathbf A}|_{n} \pm i{\mathbf D}|_{n}) \neq 0$, we apply Lemma \ref{1}, Part 1 to ${\mathbf A} \pm i{\mathbf D}|_n^0$ and obtain
$$\det({\mathbf A}|_{n}\pm i{\mathbf D}|_{n})(a_{nn} - \underline{{\mathbf a}}_{n1}({\mathbf A}|_n \pm i {\mathbf D}|_{n})^{-1}\overline{{\mathbf a}}_{1n}) = \det({\mathbf A} \pm i{\mathbf D}|_n^0). $$
Then applying Identity \eqref{id1} and equating both parts, we obtain exactly \eqref{new}.
 
Now we write a positive diagonal matrix $\mathbf D$ in the form:
$${\mathbf D} = \diag\{{\mathbf D}|_{n}, d_n\}.$$
Representing in block form \eqref{ym2}, we obtain
$${\mathbf A} \pm i{\mathbf D} = \begin{pmatrix} {\mathbf A}|_{n} & \overline{{\mathbf a}}_{1n} \\
\underline{{\mathbf a}}_{n1} & a_{nn} \end{pmatrix} \pm i\begin{pmatrix} {\mathbf D}|_{n} & \overline{{\mathbf 0}} \\
\underline{{\mathbf 0}} & d_n \end{pmatrix} = \begin{pmatrix} {\mathbf A}|_{n}\pm i{\mathbf D}|_{n} & \overline{{\mathbf a}}_{1n} \\
\underline{{\mathbf a}}_{n1} & a_{nn}\pm id_n \end{pmatrix}.$$

If $\det({\mathbf A}|_{n} \pm i{\mathbf D}|_{n}) \neq 0$, we use Lemma \ref{1}, Part 1 and expand $\det({\mathbf A} \pm i{\mathbf D})$ as follows:

$$\det({\mathbf A} \pm i{\mathbf D}) = \det({\mathbf A}|_{n}\pm i{\mathbf D}|_{n})(a_{nn} \pm id_n - \underline{{\mathbf a}}_{n1}({\mathbf A}|_{n}\pm i{\mathbf D}|_{n})^{-1}\overline{{\mathbf a}}_{1n})= $$
$$\pm id_n\det({\mathbf A}|_{n}\pm i{\mathbf D}|_{n}) + \det({\mathbf A}|_{n}\pm i{\mathbf D}|_{n})(a_{nn} - \underline{{\mathbf a}}_{n1}({\mathbf A}|_{n}\pm i{\mathbf D}|_{n})^{-1}\overline{{\mathbf a}}_{1n} )
.$$
If $a_{nn} \neq 0$, we use Equality \eqref{new} and obtain
$$ \det({\mathbf A} \pm i{\mathbf D}) = \pm id_n\det({\mathbf A}|_{n}\pm i{\mathbf D}|_{n}) + a_{nn}\det({\mathbf A}/{a_{nn}} \pm i {\mathbf D}|_n). $$
Now, let $a_{nn} = 0$. Using the properties of the adjugate matrix alongside Identity \eqref{id2}, we obtain:
$$
\begin{aligned}
\det({\mathbf A} \pm i{\mathbf D}) &= \pm id_n\det({\mathbf A}|_{n}\pm i{\mathbf D}|_{n}) + \det({\mathbf A}|_{n}\pm i{\mathbf D}|_{n})\left(- \underline{{\mathbf a}}_{n1}({\mathbf A}|_{n}\pm i{\mathbf D}|_{n})^{-1}\overline{{\mathbf a}}_{1n} \right) \\
&= \pm id_n\det({\mathbf A}|_{n}\pm i{\mathbf D}|_{n}) - \underline{{\mathbf a}}_{n1}{\rm adj}({\mathbf A}|_{n}\pm i{\mathbf D}|_{n})\overline{{\mathbf a}}_{1n} \\
&= \pm id_n\det({\mathbf A}|_{n}\pm i{\mathbf D}|_{n}) + \det({\mathbf A} \pm i{\mathbf D}|_n^0).
\end{aligned}
$$

For the case where $\det({\mathbf A}|_{n} \pm i{\mathbf D}|_{n}) = 0$, applying Lemma \ref{1}, Part 3 and Identity \eqref{id2} simplifies the result directly to:
$$
\begin{aligned}
\det({\mathbf A} \pm i{\mathbf D}) &= - \underline{{\mathbf a}}_{n1}{\rm adj}({\mathbf A}|_{n}\pm i{\mathbf D}|_{n})\overline{{\mathbf a}}_{1n} \\
&= \det({\mathbf A} \pm i{\mathbf D}|_n^0).
\end{aligned}
$$
$\square$

\section{Expanding $\det({\mathbf A} + i{\mathbf D})$: Delete/Zero algorithm}
For the remainder of the paper, we work with the $+i{\mathbf D}$ case; the $-i{\mathbf D}$ case is analogous and yields the same conclusions up to sign changes that do not affect the zero sets.
Following Johnson (\cite{JOHN5}), we introduce the following notations. Given an $n \times n$ matrix $\mathbf A$, let
\begin{equation}\label{PQ}P= P({\mathbf A}):= {\rm Re}(\det({\mathbf A}+ i{\mathbf D})); \qquad Q = Q({\mathbf A})= {\rm Im}(\det({\mathbf A}+ i{\mathbf D})).\end{equation}
Both $P$ and $Q$ are the polynomials of $n$ variables $d_1, \ \ldots, \ d_n$, with real coefficients. Hence,
\begin{equation}\label{st1} \det({\mathbf A}+ i{\mathbf D}) = P + iQ. \end{equation}

Now we start a branched process, described as follows. 
\begin{enumerate}
\item[\rm {\bf Step 0.}] Consider an $n \times n$ matrix ${\mathbf A}+ i{\mathbf D}$ and the polynomials of $n$ variables $P$ and $Q$, defined by \eqref{PQ}.
\item[\rm {\bf Step 1.}] We produce one of the following operations:
\begin{enumerate}
\item[\rm I.] Delete the last row and last column of ${\mathbf A}+ i{\mathbf D}$. The result will be an $(n-1)\times(n-1)$ submatrix ${\mathbf A}_0:={\mathbf A}|_n+ i{\mathbf D}|_n$ and the corresponding polynomials of $n-1$ variables, denoted
$$P_0:= {\rm Re}(\det({\mathbf A}|_n+ i{\mathbf D}|_n)); \qquad Q_0:= {\rm Im}(\det({\mathbf A}|_n+ i{\mathbf D}|_n)). $$
\item[\rm II.] Keep the last row and last column of ${\mathbf A}$, but set $d_n$ to zero. The result will be an $n\times n$ matrix ${\mathbf A}_{1}:={\mathbf A}+ i{\mathbf D}|_n^0$, and the corresponding polynomials of $n-1$ variables, denoted
$$P_1:= {\rm Re}(\det({\mathbf A}+ i{\mathbf D}|_n^0)); \qquad Q_1:= {\rm Im}(\det({\mathbf A}+ i{\mathbf D}|_n^0)). $$
\end{enumerate}
\item[\rm {\bf Step k+1.}] After $k$ steps, we have $2^k$ matrices ${\mathbf A}_{s_{n-k+1}\ldots s_n}$, each described by a binary vector $s = (s_{n-k+1}\ldots s_n)$ of $k$ variables $s_i \in \{0,1\}$, $i = n-k+1, \ \ldots, \ n$. Here $s_i = 0$ means the row and column with number $i$ is completely deleted from ${\mathbf A}+ i{\mathbf D}$, and $s_i = 1$ means the $i$th row and column are kept but $d_i$ is replaced with $0$.
Hence the size of ${\mathbf A}_{s_{n-k+1}\ldots s_n}$ is $n - k + \sum_{i = n-k+1}^ns_i$, and the corresponding polynomials $P_{s_{n-k+1}\ldots s_n}:= {\rm Re}(\det({\mathbf A}_{s_{n-k+1}\ldots s_n}))$ and $ Q_{s_{n-k+1}\ldots s_n}:= {\rm Im}(\det({\mathbf A}_{s_{n-k+1}\ldots s_n}))$ each depends on $n-k$ variables $d_1, \ \ldots d_{n-k}$.
We again produce one of the following operations:
\begin{enumerate}
\item[\rm I.] Delete the row and the column with index $n-k$ of ${\mathbf A}_{s_{n-k+1}\ldots s_n}$. The result will be the submatrix ${\mathbf A}_{0s_{n-k+1}\ldots s_n}$ with the size reduced by $1$, and the corresponding polynomials of $n-k-1$ variables, denoted
$$P_{0s_{n-k+1}\ldots s_n}:= {\rm Re}(\det({\mathbf A}_{0s_{n-k+1}\ldots s_n}));$$  $$Q_{0s_{n-k+1}\ldots s_n}:= {\rm Im}(\det({\mathbf A}_{0s_{n-k+1}\ldots s_n})). $$
\item[\rm II.] Keep the row and column with index $n-k$ of ${\mathbf A}_{s_{n-k+1}\ldots s_n}$, but set $d_{n-k}$ to zero. The result will be the matrix ${\mathbf A}_{1s_{n-k+1}\ldots s_n}$ of the same size, and the corresponding polynomials of $n-k-1$ variables, denoted

$$P_{1s_{n-k+1}\ldots s_n}:= {\rm Re}(\det({\mathbf A}_{1s_{n-k+1}\ldots s_n}));$$ $$Q_{1s_{n-k+1}\ldots s_n}:= {\rm Im}(\det({\mathbf A}_{1s_{n-k+1}\ldots s_n})). $$
\end{enumerate}
\item[\rm {\bf Step $n-1$.}] The algorithm halts at step $n-1$.
\end{enumerate}
In general, the algorithm performs $n$ steps in total, considering the indices from $n$ to $1$. However, after $n-1$ steps, we obtain the following result.
\begin{lemma}[Correctness of Delete/Zero algorithm]\label{p} Let ${\mathbf A} \in {\mathcal M}^{n \times n}$ and let the recursive algorithm make $n-1$ steps, processing the indices $n, \ \ldots, \ 2$. 
For each binary vector $s = (s_2, \ \ldots, \ s_n)$, $s_i \in \{0,1\}$, define the set of indices $\alpha(s)$ as follows:
$$\alpha(s):=\{i: 2 \leq i \leq n, \ s_i = 1 \}.$$
Then the matrix ${\mathbf A}_{s}$ satisfies:
$$\det({\mathbf A}_{s}) = A\begin{pmatrix}1 \cup \alpha(s) \\  1 \cup \alpha(s) \end{pmatrix} + id_1A\begin{pmatrix} \alpha(s) \\  \alpha(s) \end{pmatrix}. $$
\end{lemma}
After $n-1$ steps, we have $2^{n-1}$ matrices, each still depending on the variable $d_1$. If the algorithm proceeds one more step (Step $n$), then it processes index $1$, generating $2^n$ matrices, all with $Q_{s_1\ldots s_n} \equiv 0$, and  $P_{s_1\ldots s_n}$ be real constants, that are exactly the principal minors of 
$\mathbf A$, including the empty minor ${\mathbf A}(\emptyset):=1$.

Studying the principal minors through determinantal representations of an associated multivariate polynomial is a known technique, see \cite{AHV1}, \cite{AHV2}, where the identity 
$$\det({\mathbf A} + {\mathbf D}) = \sum_{\alpha \subseteq [n]}A\begin{pmatrix}
    \alpha \\ \alpha
\end{pmatrix}\prod_{i \in [n]\setminus \alpha}d_i$$
is given explicitly. We also note that the problem of computing all principal minors of a matrix has been studied extensively, for example, Griffin and Tsatsomeros \cite{GT1} provide an $O(2^n)$ algorithm. However, our goal is different: we seek to analyze the parameterized determinant $\det({\mathbf A} + i{\mathbf D})$ for the purpose of testing $D$-stability. The branched structure we introduce is specifically designed to track the dependence on the diagonal variables $d_i$ and to enable hierarchical testing. 

{\bf Example.}
Applying $2$ steps of Delete/Zero algorithm to a $4 \times 4$ matrix, we obtain: 
$$\det({\mathbf A} + i{\mathbf D}) = \begin{vmatrix}a_{11} + id_1 & a_{12} & a_{13} & a_{14} \\ a_{21} & a_{22} + id_{2} & a_{23} & a_{24} \\ a_{31} & a_{32} & a_{33} + id_{3} & a_{34} \\ a_{41} & a_{42} & a_{43} & a_{44} + id_{4} \end{vmatrix}$$

{\bf Step 1.} $$\det({\mathbf A}_0) = \begin{vmatrix}a_{11} + id_1 & a_{12} & a_{13}  \\ a_{21} & a_{22} + id_{2} & a_{23} \\ a_{31} & a_{32} & a_{33} + id_{3}  \end{vmatrix}; $$ $$\det({\mathbf A}_1) = \begin{vmatrix}a_{11} + id_1 & a_{12} & a_{13} & a_{14} \\ a_{21} & a_{22} + id_{2} & a_{23} & a_{24} \\ a_{31} & a_{32} & a_{33} + id_{3} & a_{34} \\ a_{41} & a_{42} & a_{43} & a_{44} \end{vmatrix}$$ 

{\bf Step 2.} $$\det({\mathbf A}_{00}) = \begin{vmatrix}a_{11} + id_1 & a_{12} &  \\ a_{21} & a_{22} + id_{2}  \end{vmatrix}; \qquad \det({\mathbf A}_{10}) = \begin{vmatrix}a_{11} + id_1 & a_{12} & a_{13}  \\ a_{21} & a_{22} + id_{2} & a_{23} \\ a_{31} & a_{32} & a_{33} \end{vmatrix};$$
$$\det({\mathbf A}_{01}) = \begin{vmatrix}a_{11} + id_1 & a_{12} &  a_{14} \\ a_{21} & a_{22} + id_{2} & a_{24} \\ a_{41} & a_{42} &  a_{44} \end{vmatrix}; $$ $$ \det({\mathbf A}_{11}) = \begin{vmatrix}a_{11} + id_1 & a_{12} & a_{13} & a_{14} \\ a_{21} & a_{22} + id_{2} & a_{23} & a_{24} \\ a_{31} & a_{32} & a_{33}  & a_{34} \\ a_{41} & a_{42} & a_{43} & a_{44} \end{vmatrix}.$$ 
Calculating the determinants and decomposing into real and imaginary parts, we obtain the following eight polynomials, collected according to the powers of $d_2$:
$$ P_{00} = - d_1d_2 + A\begin{pmatrix}
    1 & 2 \\ 1 & 2
\end{pmatrix}; \qquad Q_{00} = a_{11}d_2 + d_1a_{22}; $$
$$ P_{10} = -a_{33}d_1d_2 + A\begin{pmatrix}
    1 & 2 & 3 \\ 1 & 2 & 3
\end{pmatrix}; \qquad Q_{10} = A\begin{pmatrix}
    1 & 3 \\ 1 & 3
\end{pmatrix}d_2 + A\begin{pmatrix}
    2 & 3 \\ 2 & 3
\end{pmatrix}d_1; $$
$$P_{01} = -a_{44}d_1d_2 + A\begin{pmatrix}
    1 & 2 & 4 \\ 1 & 2 & 4
\end{pmatrix}; \qquad Q_{01} = A\begin{pmatrix}
    1 & 4 \\ 1 & 4
\end{pmatrix}d_2 + A\begin{pmatrix}
    2 & 4 \\ 2 & 4
\end{pmatrix}d_1;  $$
$$ P_{11} = -A\begin{pmatrix}
    3 & 4 \\ 3 & 4
\end{pmatrix}d_1d_2  + \det({\mathbf A}); \qquad Q_{11} = A\begin{pmatrix}
    1 & 3 & 4 \\ 1 & 3 &  4
\end{pmatrix}d_2 + A\begin{pmatrix}
    2 & 3 & 4 \\ 2 & 3 & 4
\end{pmatrix}d_1.$$

\begin{lemma}\label{explemma1} Given an arbitrary matrix ${\mathbf A} \in {\mathcal M}^{n \times n}$ and a diagonal matrix ${\mathbf D} \in {\mathcal M}^{n \times n}$, the following recurrence relations hold:
\begin{equation}\label{expsyst1}\begin{cases}P = -d_nQ_0 + P_1; \\
Q = d_nP_0 + Q_1
\end{cases}\end{equation}
for every set of positive variables $(d_1, \ \ldots, \ d_n)$.
In particular, if $a_{nn} \neq 0$ and $\det({\mathbf A}|_{n} + i{\mathbf D}|_{n}) \neq 0$ then
\begin{equation}\label{expsyst}\begin{cases}P = -d_nQ_0 + a_{nn}P({\mathbf A}/{a_{nn}}); \\
Q = d_nP_0 + a_{nn}Q({\mathbf A}/{a_{nn}}),
\end{cases}\end{equation}
where $P({\mathbf A}/{a_{nn}}) = {\rm Re}(\det({\mathbf A}/{a_{nn}}+ i{\mathbf D}|_{n}))$ and $Q({\mathbf A}/{a_{nn}}) = {\rm Im}(\det({\mathbf A}/{a_{nn}}+ i{\mathbf D}|_{n}))$
\end{lemma}
{\bf Proof.} Using the above notations, we get:
$$\det({\mathbf A}|_{n}+ i{\mathbf D}|_{n}) = P_0 + iQ_0; \qquad  \det({\mathbf A}+ i{\mathbf D}|_n^0) = P_1 + iQ_1.$$
Substituting the above expressions into Equality \eqref{expana!}, we get:
$$P + iQ = id_n(P_0 + iQ_0) + P_1 + iQ_1 = P_1 - d_nQ_0 +i(d_nP_0 + Q_1). $$
Equating real and imaginary parts of both the sides, we complete the proof.

For the special case, when $a_{nn} \neq 0$ and $\det({\mathbf A}|_{n} + i{\mathbf D}|_{n}) \neq 0$, we have
$$\det({\mathbf A}/{a_{nn}}+ i{\mathbf D}|_{n}) = P({\mathbf A}/{a_{nn}}) + iQ({\mathbf A}/{a_{nn}}).$$
Substituting the above expressions into Equality \eqref{expan}, we get:
$$P({\mathbf A}) + iQ({\mathbf A}) = id_n(P({\mathbf A}|_{n}) + iQ({\mathbf A}|_{n})) + a_{nn}(P({\mathbf A}/{a_{nn}}) + iQ({\mathbf A}/{a_{nn}})) = $$ $$a_{nn}P({\mathbf A}/{a_{nn}}) - d_nQ({\mathbf A}|_{n}) +i(d_nP({\mathbf A}|_{n})) + a_{nn}Q({\mathbf A}/{a_{nn}}). $$
Again, equating real and imaginary parts of both the sides, we complete the proof.
$\square$

Now consider the matrix obtained after $k$ steps of the delete/zero algorithm. Put $s_{n+1} := \emptyset$. For a given binary string ${s} = \{s_{n-k+1}, \ \ldots, \ s_n\}$ of length $k$, use the notation $0s$ for the binary string $0s_{n-k+1} \ldots s_n$ of length $k+1$, and, respectively, $1s$ for the string $1s_{n-k+1} \ldots s_n$.
\begin{lemma}[Recurrence relations for $P_{s}$ and $Q_{s}$]\label{rec} Given an arbitrary matrix ${\mathbf A} \in {\mathcal M}^{n \times n}$ and a diagonal matrix ${\mathbf D} \in {\mathcal M}^{n \times n}$. For a binary string ${s} = \{s_{n-k+1}, \ \ldots, \ s_n\}$, $s_i \in \{0,1\}$, $0 \leq k \leq n$, let ${\mathbf A}_{s}$ be one of the matrices, obtained from $\mathbf A$ after $k$ steps of delete/zero algorithm. Then the following recurrence relations hold:
\begin{equation}\label{expsyst1}\begin{cases}P_{s} = -d_{n-k}Q_{0s} + P_{1s}; \\
Q_{s} = d_{n-k}P_{0s} + Q_{1s}
\end{cases}\end{equation}
for every set of positive variables $(d_1, \ \ldots, \ d_{n-k})$.
\end{lemma}
The proof obviously follows from Lemma \ref{explemma1}.

\section{Basic results on $D$-stability}
First, recall the following basic results (see, e.g. \cite{JOHN1}).

\begin{lemma}[Elementary properties of $D$-stable matrices]\label{el} If $\mathbf A$ is $D$-stable then $\mathbf A$ is nonsingular and each of the following matrices are also $D$-stable:
\begin{enumerate}
\item[\rm 1.] ${\mathbf A}^T$
\item[\rm 2.] ${\mathbf A}^{-1}$
\item[\rm 3.] ${\mathbf P}^T{\mathbf A}{\mathbf P},$ where $\mathbf P$ is any permutation matrix.
\item[\rm 4.] ${\mathbf D}{\mathbf A}{\mathbf E}$, where $\mathbf D$, $\mathbf E$ are positive diagonal matrices.
\end{enumerate}
\end{lemma}

Here, we recall the following equivalent conditions of matrix $D$-stability (see \cite{JOHN5}, p. 89, Corollary 2).

\begin{theorem}[Johnson]\label{stabcond} Let ${\mathbf A} \in {\mathcal M}^{n \times n}$ be stable. Then the following conditions are equivalent.
\begin{enumerate}
\item[\rm (i)] ${\mathbf A}$ is $D$-stable.
\item[\rm (ii)] $\det({\mathbf A} \pm i{\mathbf D}) \neq 0$ for every positive diagonal matrix $\mathbf D$.
\end{enumerate}
\end{theorem}
The condition (ii) shows that ${\mathbf D}{\mathbf A}$ has no eigenvalues on the imaginary axis for any positive diagonal matrix $\mathbf D$.

Let us recall the following definitions from matrix theory.

{\bf Definition 5.} A matrix ${\mathbf A} \in {\mathcal M}^{n \times n}$ is called:
\begin{enumerate}
\item[-] a {\it $P$-matrix ($P_0$-matrix)} if all its principal minors are positive (respectively, nonnegative), i.e the inequality $A \left(\begin{array}{ccc}i_1 & \ldots & i_k \\ i_1 & \ldots & i_k \end{array}\right) > 0$ (respectively, $\geq 0$)
holds for all $(i_1, \ \ldots, \ i_k), \ 1 \leq i_1 < \ldots < i_k \leq n$, and all $k, \ 1 \leq k \leq n$.
\item[-] a {\it $P_0^+$-matrix} if it is a $P_0$-matrix and, in addition, the inequality
$$\sum_{(i_1, \ldots, i_k)}A \begin{pmatrix}i_1 & \ldots & i_k \\ i_1 & \ldots & i_k \end{pmatrix} > 0 $$
holds for all $k, \ 1 \leq k \leq n$, i.e. the sums of all principal minors of every fixed order $k$ are positive.
\end{enumerate}

A important necessary for $D$-stability condition was proved in \cite{QR}: {\it if $\mathbf A$ is $D$-stable then ${\mathbf A}$ is a $P_0^+$-matrix. Consequently, if $\mathbf A$ is positive $D$-stable then ${\mathbf A}$ is a $P_0^+$-matrix.}

Let us recall the following statement by Johnson (see Theorem 2 and Corollary 3 in \cite{JOHN5}).
\begin{theorem}\label{Johhn} Let ${\mathbf A} \in {\mathcal M}^{n \times n}$ be stable. Then the following conditions are equivalent.
\begin{enumerate}
\item[\rm (i)] ${\mathbf A}$ is $D$-stable.
\item[\rm (ii)] The system of polynomial equations \begin{equation}\label{sy}
\begin{cases}P = 0; \\
Q = 0 \end{cases}
\end{equation}
has no positive solutions $(d_1, \ \ldots, \ d_n)$, $d_i > 0$, $i = 1, \ \ldots, \ n$.
\item[\rm (iii)] The polynomial equation 
$$F(d_1, \ \ldots, \ d_n) = 0,$$
where
\begin{equation}\label{ones}
F:=P^2 + Q^2,
\end{equation}
has no positive solutions $(d_1, \ \ldots, \ d_n)$, $d_i > 0$, $i = 1, \ \ldots, \ n$ (or, equivalently, $F > 0$ whenever $d_i > 0$, $i = 1, \ \ldots, \ n$).
\end{enumerate}
\end{theorem}

\section{Analyzing Step 1}

For a given binary string ${s} = (s_{n-k+1}\ldots s_n)$, $s_i \in \{0,1\}$, $0 \leq k \leq n$, we introduce the notations:
$$F_{s} : = P^2_{s} + Q^2_{s}.$$

Using Relations \ref{expsyst1}, we get the following recursive formula:
$$F = (-d_nQ_0 + P_1)^2 +
(d_nP_0 + Q_1)^2 = $$
$$d_n^2(Q^2_0 + P^2_0) + 2d_n(P_0Q_1 - Q_0P_1) + Q^2_1 + P^2_1$$
and, finally
\begin{equation}\label{Jrec}
    F = d_n^2F_0 + 2d_n(P_0Q_1 - Q_0P_1) + F_1,
\end{equation}

Now introduce the following notations. Given two binary strings ${s}$ and ${t}$ of the same length, define
$$F(s, {t}):={\rm Re}(\overline{\det({\mathbf A}_{s})}\det({\mathbf A}_{t})); \qquad  G(s, {t}) :={\rm Im}(\overline{\det({\mathbf A}_{s})}\det({\mathbf A}_{t})). $$
Easy calculations show that
$$ \overline{\det({\mathbf A}_{s})}\det({\mathbf A}_{t}) = (P_{s}- iQ_{s})(P_{t} + iQ_{t}) =$$
$$(P_{s}P_{t} + Q_{s}Q_{t}) + i(P_{s}Q_{t} - Q_{s}P_{t}).$$
Hence \begin{equation}\label{FGs}F(s, t)= P_{s}P_{t} + Q_{s}Q_{t}; \qquad G(s, t)= P_{s}Q_{t} - Q_{s}P_{t}.
\end{equation}
and for each binary string ${s}$: $$F_s = F(s,s) = P_s^2 + Q_s^2  = \overline{\det({\mathbf A}_{s})}\det({\mathbf A}_{s}) = | \det({\mathbf A}_{s})|^2$$
In particular,
\begin{equation}\label{FG}F(0, 1)= P_0P_1 + Q_0Q_1, \qquad 
    G(0, 1)= P_0Q_1 - Q_0P_1. 
\end{equation}

Using the notations $0{s} = {0s_{n-k+1}\ldots s_n}$ and $1{s} = {1s_{n-k+1}\ldots s_n}$, we obtain
\begin{equation}\label{JS}
    F_{s} = d_{n-k}^2F_{0{s}} + 2d_{n-k}G(0{s},1{s}) + F_{1{s}}.\end{equation}
Now state and prove the following (necessary and sufficient) criterion of $D$-stability in terms of the polynomials $F$ and $G$ at first recursion step. 
\begin{theorem}\label{Step1}
    Let ${\mathbf A} \in {\mathcal M}^{n \times n}$ be stable. Then $\mathbf A$ is $D$-stable if and only if the following two conditions holds.
    \begin{enumerate}
        \item[ \rm 1.] Non-degenerate case. For all the values $d_1, \ \ldots, \ d_{n-1}$ that satisfy $$F_0(d_1, \ \ldots, d_{n-1})> 0,$$ the equation
        $$F(0, 1)(d_1, \ \ldots, \ d_{n-1}) = 0, $$
        does not have any positive solutions $(d^0_{1}, \ \ldots, \ d^0_{n-1})$, for which $$G(0, 1)(d^0_{1}, \ \ldots, \ d^0_{n-1}) < 0.$$
        \item[ \rm 2.] Degenerate case. All the positive solutions $d^0_{1}, \ \ldots, \ d^0_{n-1}$ of the equation $$F_0(d^0_{1}, \ \ldots, d^0_{n-1})=0$$ satisfy $F_1(d^0_{1}, \ \ldots, d^0_{n-1})>0$.
    \end{enumerate}
\end{theorem}
{\bf Proof.} We establish the equivalence by proving both directions. 
We use the notation $\text{(i)} \Rightarrow \text{(ii)}$ to mean that $D$-stability 
implies Conditions 1 and 2, and $\text{(ii)} \Rightarrow \text{(i)}$ for the converse.

$(i) \Rightarrow (ii)$. Assume on the contrary, let $\mathbf A$ be $D$-stable, and one of the conditions, either 1 or 2 fails. 

Let Condition 1 fail. Then the equation $F(0,1) = 0$ has a positive solution $d_1^0, \ \ldots, \ d_{n-1}^0$, which satisfies $$F_0(d_1^0, \ \ldots, \ d_{n-1}^0) > 0\qquad \mbox{and} \qquad G(0,1)(d_1^0, \ \ldots, \ d_{n-1}^0) < 0.$$ First consider the case when both $P_0$ and $Q_0$ do not vanish on $d_1^0, \ \ldots, \ d_{n-1}^0$. Then $P_0P_1 + Q_0Q_1 = 0$ implies $\frac{P_1}{Q_0} = - \frac{Q_1}{P_0}$, while the condition $P_0Q_1 - Q_0P_1 < 0$ (after dividing by $P_0^2Q_0^2$) shows that $\frac{P_1}{Q_0} = - \frac{Q_1}{P_0} > 0$ on $d_1^0, \ \ldots, \ d_{n-1}^0$. Hence we can put $$d^0_{nn}:= \frac{P_1}{Q_0}(d_1^0, \ \ldots, \ d_{n-1}^0) = - \frac{Q_1}{P_0}(d_1^0, \ \ldots, \ d_{n-1}^0) > 0$$ and substitute it into System \eqref{expsyst1}, obtaining:
$$\begin{cases}P(d_1^0, \ \ldots, \ d_n^0) = 0; \\
Q(d_1^0, \ \ldots, \ d_n^0) = 0 \end{cases}
$$
This contradicts Theorem \ref{Johhn}, (iii).

Now consider the case, when one of $P_0$ or $Q_0$ equals zero, while the other is not. Consider the case when $P_0 = 0,$ $Q_0 \neq 0$ (the opposite case is considered analogically). Then $d_n :=-\frac{Q_0}{P_1} > 0$, which vanishes the first equation in System \eqref{expsyst1}, while the second one is identical zero. This again contradicts Theorem \ref{Johhn}, (iii). 

Finally, let Condition 2 fails. Then there is a set of positive numbers $(d_1^0, \ \ldots, \ d_{n-1}^0)$, such that $$\begin{cases}F_0(d_1^0, \ \ldots, \ d_{n-1}^0) = 0; \\ F_1(d_1^0, \ \ldots, \ d_{n-1}^0) = 0\end{cases} \quad \mbox{or, equivalently,} \quad \begin{cases}P_0(d_1^0, \ \ldots, \ d_{n-1}^0) = 0; \\ Q_0(d_1^0, \ \ldots, \ d_{n-1}^0) = 0; \\ P_1(d_1^0, \ \ldots, \ d_{n-1}^0) = 0; \\ Q_1(d_1^0, \ \ldots, \ d_{n-1}^0) = 0.\end{cases}$$
Substituting into \eqref{Jrec}, we obtain $$F(d_1^0, \ \ldots, \ d_{n-1}^0, \ d_n) = 0$$
for every $d_n \in {\mathbb R}$, which contradicts Theorem \ref{Johhn}, (iii).

$(ii) \Rightarrow (i)$ By Theorem \ref{Johhn}, (ii) the matrix $\mathbf A$ is $D$-stable if and only if System \eqref{sy} has no positive solutions. Assume on the contrary, let $\mathbf A$ be not $D$-stable, that is, $$\begin{cases}P(d^0_{1}, \ \ldots, \ d^0_{n}) = 0; \\
Q(d^0_{1}, \ \ldots, \ d^0_{n}) = 0 \end{cases}$$ for some set of positive values $(d^0_{1}, \ \ldots, \ d^0_{n})$. Now consider the following two cases.

Case I. Let $F_0(d^0_{1}, \ \ldots, \ d^0_{n-1}) = 0$. Then $P_0$ and $Q_0$ vanish simultaneously. From System \eqref{expsyst1}, we obtain that $P_1$ and $Q_1$ also vanish simultaneously for the same positive set $(d^0_{1}, \ \ldots, \ d^0_{n-1})$ and Condition 2 fails.

Case II. Now let $F_0(d^0_{1}, \ \ldots, \ d^0_{n-1}) > 0$. Then either $P_0 \neq 0$ or $Q_0 \neq 0$ on $(d^0_{1}, \ \ldots, \ d^0_{n-1})$. Thus we can extract $d_n^0$ from either the first or the second equation of System \eqref{expsyst1}:
\begin{equation}\label{dnn}0 < d_n^0 = \frac{P_1}{Q_0} \qquad \mbox{or} \qquad 0 < d_n^0 = - \frac{Q_1}{P_0}.\end{equation}
Substituting $d_n^0$ into the other equation, we get that the positive set $(d^0_{1}, \ \ldots, \ d^0_{n-1})$ is a solution of the equation $$F(0,1) = P_1P_0 + Q_1Q_0 = 0.$$
Then, multiplying the nonzero fraction from \eqref{dnn}, by $Q_0^2$, if $Q_0 \neq 0$ (or, respectively, by $P_0^2$, if $P_0 \neq 0$), we get
$$Q_0^2d_n^0 = P_1Q_0 \geq 0; \qquad P_0^2d_n^0 = - Q_1P_0 > 0$$
with at least one inequality being strict. By summation, we obtain 
$$d_n^0(Q_0^2 + P_0^2) = -G(0, 1) > 0$$
for the same positive set $(d^0_{1}, \ \ldots, \ d^0_{n-1})$. Then Condition 1 also fails.
$\square$

It is easy to see that the condition $F(0, 1) = 0$ is equivalent to the solvability of the system 
$$\begin{cases}d_nP_0 = - Q_1 \\ d_nQ_0 = P_1 \end{cases}$$ Multiplying the equations by $P_0$ and $Q_0$ respectively, and adding, we obtain:
$$d_n(P_0^2 + Q_0^2) = P_1Q_0 - Q_1P_0 = -G(0, 1). $$
Hence in the case $F_0 > 0$, we obtain $d_n = - \frac{G(0, 1)}{F_0}$. 
\begin{corollary}[Sufficient condition for $D$-stability]\label{suff}
  Let $\mathbf A$ be a stable $n\times n$ matrix. Then for $\mathbf A$ to be $D$-stable, it is sufficient that one of the following conditions holds:
  \begin{enumerate}
      \item[\rm 1.] the inequality $F(0,1) > 0$ holds for every positive set $(d_1, \ \ldots, \ d_{n-1})$;
      \item[\rm 2.] the inequality $G(0,1) > 0$ holds for every positive set $(d_1, \ \ldots, \ d_{n-1})$.
      \end{enumerate}
\end{corollary}
{\bf Proof.} First note that if $F_0(d_1 \ldots d_{n-1}) = 0$ (degenerate case) then both $F(0,1) = 0$ and $G(0,1) = 0$. 
Hence if either Condition $1$ or Condition $2$ of Corollary \ref{suff} holds, then $F_0 > 0$ for every positive set $(d_1, \ \ldots, \ d_{n-1})$ (non-degenerate case of Theorem \ref{Step1}). 

Let Condition 1 holds. Then the multi-variable polynomial $F(0,1)$ preserves its sign (positive) everywhere on the positive orthant $$\{(d_1, \ \ldots, \ d_{n-1}): \ d_i > 0, \ i = 1, \ \ldots, \ n-1\}.$$ Thus the equation $F(0,1) = 0$ do not have any positive solutions $(d_1^0, \ \ldots, \ d_{n-1}^0)$ and $\mathbf A$ is $D$-stable by Theorem \ref{Step1}.

Now let Condition 2 holds. Then obviously $G(0,1) > 0$ for all positive solutions $(d_1^0, \ \ldots, \ d_{n-1}^0)$ of the equation $F(0,1) = 0$ and again $\mathbf A$ is $D$-stable by Theorem \ref{Step1}.
$\square$

Note that both the polynomials $F(0,1)$ and $G(0,1)$ are of degree no more than 2 in each variable $d_1, \ \ldots, \ d_{n-1}$.

\section{Analyzing Step 2}
After Step 1, we have two resulting matrices, ${\mathbf A}_0$ and ${\mathbf A}_1$. Processing index $n-2$ and applying Lemma \ref{rec} to each branch, we obtain the following recurrence relations:
\begin{equation}\label{expsyst2a}\begin{cases}P_0 = - d_{n-1}Q_{00} + P_{10}; \\
Q_0 = d_{n-1}P_{00} + Q_{10}
\end{cases}\end{equation}
and 
\begin{equation}\label{expsyst2b}\begin{cases}P_1 = - d_{n-1}Q_{01} + P_{11}; \\
Q_1 = d_{n-1}P_{01} + Q_{11}
\end{cases}\end{equation}

Substituting into Formulas \eqref{FG} and collecting with respect to the powers of $d_{n-1}$, we obtain for $F(0,1)$, the following quadratic expression in variable $d_{n-1}$:
$$F(0,1) = (P_{00}P_{01} + Q_{00}Q_{01})d_{n-1}^2 + $$
$$(P_{00}Q_{11} - Q_{00}P_{11} + P_{01}Q_{10} - Q_{01}P_{10})d_{n-1} + P_{10}P_{11} +Q_{10}Q_{11}.$$
Using the notations \eqref{FGs}, we obtain:
$$F(00,01) = P_{00}P_{01} + Q_{00}Q_{01}; \qquad F(10,11) = P_{10}P_{11} +Q_{10}Q_{11}; $$
$$G(00, 11) = P_{00}Q_{11} - Q_{00}P_{11}; \qquad G(01,10)= P_{01}Q_{10} - Q_{01}P_{10} $$
and re-write the expression as follows:
\begin{equation}\label{Fexp}
    F(0,1) = F(00,01)d^2_{n-1} + (G(00,11) - G(10,01)) d_{n-1} + F(10,11).
\end{equation}
Analogically, for $G(0,1)$, we obtain:
$$G(0,1) = (P_{00}Q_{01} - Q_{00}P_{01})d^2_{n-1} +$$
$$(P_{10}P_{01} + Q_{10}Q_{01} -(P_{00}P_{11} + Q_{00}Q_{11}))d_{n-1} + P_{10}Q_{11} - Q_{10}P_{11}.$$
Since
$$G(00,01) = P_{00}Q_{01} - Q_{00}P_{01}; \qquad G(10,11) = P_{10}Q_{11} - Q_{10}P_{11}; $$
$$F(00, 11) = P_{00}P_{11} + Q_{00}Q_{11}; \qquad F(10,01)= P_{01}P_{10} + Q_{01}Q_{10} $$
we obtain:
\begin{equation}\label{Gexp}
    G(0,1) = G(00,01)d^2_{n-1} + ( F(10,01) - F(00,11)) d_{n-1} + G(10,11).
\end{equation}

\subsection{Non-degenerate case}
Let us analyze the case, when the system $$\begin{cases} F_0(d_1, \ \ldots, d_{n-1})> 0; \\ F(0, 1)(d_1, \ \ldots, \ d_{n-1}) = 0, \end{cases}$$
        does not have any positive solutions $(d^0_{1}, \ \ldots, \ d^0_{n-1})$, for which $$G(0, 1)(d^0_{1}, \ \ldots, \ d^0_{n-1}) < 0.$$
Since $F_0 = P_0^2 + Q_0^2 \neq 0$, we have two cases: $Q_0 \neq 0$ and $Q_0 = 0, \ P_0 \neq 0$.
\begin{enumerate}
\item[\rm 1.] First consider the case when $Q_0 \neq 0$.  By $Q_0 = d_{n-1}P_{00} + Q_{10}$, the necessary condition is $P^2_{00} + Q^2_{10} \neq 0$. Extracting $Q_1 = - \frac{P_0 P_1}{Q_0}$ and substituting into
$G(0,1) = P_0Q_1 - Q_0P_1$, we get $G(0,1) = - \frac{P_1}{Q_0}F_0$. Since $F_0 > 0$, the inequality $G(0,1) < 0$ is equivalent to $P_1Q_0 > 0$. Hence we obtain the system
$$\begin{cases}F(0,1) = 0; \\  P_1Q_0 > 0 \end{cases}$$
Substituting from \eqref{expsyst2a} - \eqref{Fexp}: 
\begin{equation}\label{bad}
    \begin{cases} F(00,01)d^2_{n-1} + (G(00,11) - G(10,01)) d_{n-1} + F(10,11) = 0; \\  (-d_{n-1}Q_{01} + P_{11})(d_{n-1}P_{00} + Q_{10}) > 0 \end{cases}\end{equation} 

To check, when System \eqref{bad} have no positive solutions, it is enough to check, when the quadratic (in general case) polynomial $F(0,1)(d_{n-1})$ have no zeroes in the set
$$S = (0, +\infty)\cap\{(-d_{n-1}Q_{01} + P_{11})(d_{n-1}P_{00} + Q_{10}) > 0 \}. $$

Note, that if $F(00, 01) = 0$ then the equation $F(0,1) = 0$ is linear and has:
\begin{enumerate}
    \item[-] one zero (namely, $d_{n-1} = -\frac{F(10, 11)}{G(00,11) - G(10,01)}$) if $G(00,11) - G(10,01) \neq 0$, 
    \item[-] no zeroes if $G(00,11) - G(10,01) = 0$ and $F(10, 11) \neq 0$, 
     \item[-] is identically zero if all the coefficients are zero.  
\end{enumerate}

First assume that $Q_{01}P_{00} \neq 0$. In this case, we obtain the inequality
$$-Q_{01}P_{01}\left(d_{n-1} - \frac{P_{11}}{Q_{01}}\right)\left(d_{n-1} + \frac{Q_{10}}{P_{00}}\right) > 0.$$
Denote $m$ and $M$ the minimum and maximum from the values $\frac{P_{11}}{Q_{01}}$ and $-\frac{Q_{10}}{P_{00}}$, respectively. Note that 
$$\frac{P_{11}}{Q_{01}} + \frac{Q_{10}}{P_{00}} = \frac{P_{11}P_{00} + Q_{01}Q_{10}}{Q_{01}P_{00}} \geq 0 \qquad \mbox{iff} \qquad M = \frac{P_{11}}{Q_{01}}.$$

\begin{tikzpicture}
\begin{axis} [
    axis lines = middle,
    xlabel = {$d_{n-1}$},
    ylabel = {Value},
    xmin = -0.5, xmax = 5.5,
    ymin = -4, ymax = 6,
    domain = -0.5:5.5,
    samples = 100,
    axis line style={->},
    every axis x label/.style={at={(current axis.right of origin)},anchor=north west},
    every axis y label/.style={at={(current axis.above origin)},anchor=north east},
    xtick={1, 1.5, 2.5, 3.5, 4},
    xticklabels={$d_{n-1}^{(1)}$, $m$, $d_v$, $M$, $d_{n-1}^{(2)}$},
    ytick=\empty,
    enlargelimits=false
]

    \path[name path=top] (axis cs:-0.5,6) -- (axis cs:1.5,6) 
        -- plot[domain=1.5:3.5] (\x, {-2*(\x-1.5)*(\x-3.5)}) 
        -- (axis cs:3.5,6) -- (axis cs:5.5,6);
        
    \path[name path=bottom] (axis cs:-0.5,0) -- (axis cs:5.5,0);

    \addplot[green!15!white] fill between[of=top and bottom, soft clip={domain=1.5:3.5}];

    \addplot[
        blue, 
        thick, 
        name path=Fcurve
    ] {1.2*(x-1)*(x-4)} node[pos=0.9, right] {$F(0,1)$};

    \draw[dashed, gray] (axis cs:1.5,0) -- (axis cs:1.5,2);
    \draw[dashed, gray] (axis cs:3.5,0) -- (axis cs:3.5,2);
    
    \draw[dashed, gray] (axis cs:2.5,0) -- (axis cs:2.5,-2.7);

    \node[circle, fill=black, inner sep=1.5pt] at (axis cs:1.5,0) {};
    \node[circle, fill=black, inner sep=1.5pt] at (axis cs:3.5,0) {};

    \node[circle, fill=blue, inner sep=1.5pt] at (axis cs:1,0) {};
    \node[circle, fill=blue, inner sep=1.5pt] at (axis cs:4,0) {};
    \node[circle, fill=purple, inner sep=1.5pt] at (axis cs:2.5,0) {};
    
    \node[circle, fill=red, inner sep=1.5pt, label={below right:\small $F_v$}] at (axis cs:2.5,-2.7) {};

    \node[align=center, font=\small, text=green!50!black] at (axis cs:2.5, 1.2) {$S=\{P_1Q_0 > 0\}$};

\end{axis}
\end{tikzpicture}

Through the basic analysis of zero location of the quadratic polynomial $F(0, 1)(d_{n-1})$, we obtain
\begin{enumerate}
    \item[\rm 1.] $Q_{01}P_{00} > 0$: \begin{enumerate}
        \item[\rm 1a.] If $M \leq 0$ then $S = \emptyset$: automatically no zeroes in $S$. Hence this case corresponds to the system
        $$\begin{cases}Q_{01}P_{00} > 0; \\ P_{11}Q_{01} \leq 0; \\ Q_{10}P_{00} \geq 0. \end{cases}$$
\item[\rm 1b.] If $m \leq 0 < M$: $S = (0, M)$. From the condition $mM \leq 0$ we get $-\frac{P_{11}Q_{10}}{Q_{01}P_{00}} \leq 0$ and the system $$\begin{cases} Q_{01}P_{00} > 0; \\ P_{11}Q_{10} \geq 0. \end{cases}$$ 
  For zero localization outside $S$, we obtain the following necessary and sufficient conditions: 
  for the values at the endpoints of the interval (taking into account that $F(0,1)(0) = F(10,11)$) 
  $$F(10,11) F(0,1)(M) > 0,$$
  and for the parabola vertex coordinates $(d_v, F_v)$
  $$[d_v \notin (0, M) \ \mbox{or} \ (d_v \in (0, M) \ \mbox{and} \ F_v   F(10,11) > 0)],$$
  where $$d_v = - \frac{G(00,11) - G(10,01)}{ 2F(00,01)};$$ $$F_v = - \frac{D}{4F(00,01)};$$
  $$D = (G(00,11) - G(10,01))^2 - 4F(00,01)F(10,11). $$

  In the case $F(00, 01) = 0$, the conditions are $$ -\frac{F(10, 11)}{G(00,11) - G(10,01)} < 0$$ or $$ -\frac{F(10, 11)}{G(00,11) - G(10,01)} > M.$$ 

\item[\rm 1c.] If $0 \le m < M$: $S = (m, M)$.  
 The necessary and sufficient conditions are as follows: $$F(0,1)(m) F(0,1)(M) > 0$$ and $$[d_v \notin (m, M) \ \mbox{or} \ (d_v \in (m, M) \ \mbox{and} \ F_v   F(0,1)(m) > 0)].$$

 In the case $F(00, 01) = 0$, the conditions are $$ -\frac{F(10, 11)}{G(00,11) - G(10,01)} < m$$ or $$ -\frac{F(10, 11)}{G(00,11) - G(10,01)} > M.$$
 \end{enumerate}
\item[\rm 2.] $Q_{01}P_{00} < 0$:
 \begin{enumerate}
        \item[\rm 2a.] If $M < 0$ then $S = (0, +\infty)$: if $F(0,1)$ has any positive zeroes, they are automatically in $S$. The necessary and sufficeient condition for $F(0,1)$ to have no positive zeroes are as follows:
        $$F(00,01)(G(00,11) - G(10,01)) \geq 0; \qquad F(00,01)F(10,11) > 0 $$ or
        $$F(00,01)(G(00,11) - G(10,01)) < 0; \qquad D < 0. $$
        In the case $F(00, 01) = 0$, the conditions are $$ -\frac{F(10, 11)}{G(00,11) - G(10,01)} < 0.$$
\item[\rm 2b.] If $m < 0 \leq M$: $S = (M, \infty)$.  
  For zero localization outside $S$, we obtain the following necessary and sufficient conditions:
  $$F(00, 01) F(0,1)(M) > 0$$ and $$[d_v \leq M \ \mbox{or} \ (d_v > M \ \mbox{and} \ D < 0)].$$ 
  In the case $F(00, 01) = 0$, the conditions are $$ -\frac{F(10, 11)}{G(00,11) - G(10,01)} < M.$$
\item[\rm 2c.] If $0 \le m < M$: $S = (0, m)\cup(M, +\infty)$.  
 The necessary and sufficient conditions are the intersection of the above cases: $$F(10,11) F(0,1)(m) > 0;$$  $$[d_v \notin (0, m) \ \mbox{or} \ (d_v \in (0, m) \ \mbox{and} \ F_v   F(0,1)(m) > 0)];$$
 and
   $$F(00, 01) F(0,1)(M) > 0;$$ 
 $$[d_v \leq M \ \mbox{or} \ (d_v > M \ \mbox{and} \ D < 0)].$$
 \end{enumerate} In the case $F(00, 01) = 0$, the conditions are $ -\frac{F(10, 11)}{G(00,11) - G(10,01)} < 0$ or $m < -\frac{F(10, 11)}{G(00,11) - G(10,01)} < M$.
 \item[\rm 3.] Consider the third case when either $Q_{01} = 0$ or $P_{00} = 0$ (if $Q_{01} = P_{00} = 0$ then $F(00,01) = 0$ and the equation $F(0,1) = 0$ becomes linear).
 \begin{enumerate}
\item[\rm 3a.] Let $Q_{01} = 0$. In this case, $M = -\frac{Q_{10}}{P_{00}}$.
Let $P_{11}P_{00} < 0$. In this case, if $M \leq 0$ then $S = \emptyset$.
 If $M > 0$ then $S = (0, M)$.  
 The necessary and sufficient conditions are as follows: 
 $$F(10,11) F(0,1)(M) > 0$$ and $$[d_v \notin (0, M) \ \mbox{or} \ (d_v \in (0, M) \ \mbox{and} \ F_v   F(0,1)(M) > 0)].$$
 In the case $F(00, 01) = 0$, the conditions are $$ -\frac{F(10, 11)}{G(00,11) - G(10,01)} < 0$$ or $$ -\frac{F(10, 11)}{G(00,11) - G(10,01)} > M.$$

If $P_{11}P_{00} > 0$, then $M \leq 0$ implies $S = (0, + \infty)$. Hence if $F(0,1)$ has any positive zeroes, they are automatically in $S$. The necessary and sufficient conditions for $F(0,1)$ to have no positive zeroes are given above in Part 2a.

If $M > 0$, $S = (M, \infty)$. 
  For zero localization outside $S$, we obtain the following necessary and sufficient conditions:
  $$F(00,01) F(0,1)(M) > 0$$ and $$[d_v \leq M \ \mbox{or} \ (d_v > M \ \mbox{and} \ D < 0)].$$ 
  In the case $F(00, 01) = 0$, the conditions are $ -\frac{F(10, 11)}{G(00,11) - G(10,01)} < M$.
  
\item[\rm 3b.] Let $P_{00} = 0$, then $M = \frac{P_{11}}{Q_{01}}$. For $Q_{10}Q_{01} < 0$, we have the following options: if $M \leq 0$ then $S = \emptyset$ and, respectively, if $M > 0$ then $S = (0, M)$.  
 The necessary and sufficient conditions are as follows: $$F(10,11) F(0,1)(M) > 0$$ and $$[d_v \notin (0, M) \ \mbox{or} \ (d_v \in (0, M) \ \mbox{and} \ F_v   F(0,1)(M) > 0)].$$
 In the case $F(00, 01) = 0$, the conditions are $$ -\frac{F(10, 11)}{G(00,11) - G(10,01)} < 0$$ or $$ -\frac{F(10, 11)}{G(00,11) - G(10,01)} > M.$$

Let $Q_{10}Q_{01} > 0$, then $M \leq 0$ implies $S = (0, + \infty)$. if $F(0,1)$ has any positive zeroes, they are automatically in $S$, so for the necessary and sufficient conditions see Part 2a.
If $M > 0$, $S = (M, \infty)$. 
  For zero localization outside $S$, we obtain the following necessary and sufficient conditions:
  $$F(00,01) F(0,1)(M) > 0$$ and $$[d_v \leq M \ \mbox{or} \ (d_v > M \ \mbox{and} \ D < 0)].$$ 
  In the case $F(00, 01) = 0$, the conditions are $-\frac{F(10, 11)}{G(00,11) - G(10,01)} < M$.

 Let $P_{00} = Q_{01} = 0$ then $S = \emptyset$ if $P_{11}Q_{10} \leq 0$ and $S = (0, + \infty)$ if $P_{11}Q_{10} > 0$. For $F(00, 01) = 0$, the conditions are $ -\frac{F(10, 11)}{G(00,11) - G(10,01)} < 0$.
 \end{enumerate}
\end{enumerate}
\item[\rm 2.] $Q_0 = 0$. Then $F(0,1) = 0$ and $F_0 > 0$ imply $P_0 \neq 0$ and $P_1 = 0$. Then $G(0,1) = P_0Q_1 < 0$ and we have the system:
\begin{equation}\label{2}\begin{cases}Q_0 = P_1 = 0; \\ P_0Q_1 < 0 \end{cases}\end{equation}

By substitution from \eqref{expsyst2a}-\eqref{expsyst2b}, and taking into account that $$P_0 = - d_{n-1}Q_{00} + P_{10} \neq 0$$ we obtain:  
\begin{equation}\label{2}\begin{cases}Q_{00}^2 + P^2_{10} \neq 0; \\ d_{n-1}P_{00} + Q_{10} = - d_{n-1}Q_{01} + P_{11} = 0; \\ (- d_{n-1}Q_{00} + P_{10})(d_{n-1}P_{01} + Q_{11}) < 0. \end{cases}\end{equation}
\end{enumerate}

\begin{lemma}\label{Q0} System \eqref{2} admits a positive solution if and only if the following conditions hold:
\begin{enumerate}
    \item[\rm 1.] For the non-degenerate case $P_{00} \neq 0$, 
$$\begin{cases}P_{00}P_{11}+Q_{10}Q_{01} = 0; \\
P_{11}Q_{01} - P_{00}Q_{10} > 0; \\
(P_{10}P_{00} + Q_{10}Q_{00})(P_{00}Q_{11} - Q_{10}P_{01}) < 0. \\
\end{cases} $$
\item[\rm 2.] For the degenerate case $P_{00} = 0$, we have
$$\begin{cases}P_{11}Q_{01} > 0; \\ (P_{10}Q_{01} - P_{11}Q_{00})(P_{11}P_{01} + Q_{11}Q_{01})< 0  \end{cases}$$
\end{enumerate}
In both cases, the solution is unique and given by the formula 
$$d_{n-1} = \frac{P_{11}Q_{01} - P_{00}Q_{10}}{P^2_{00} + Q^2_{01}} > 0.$$
\end{lemma}
{\bf Proof.} Solving the equations $Q_0 = 0$ and $P_1 = 0$ with respect to $d_{n-1}$ and equating the solutions, we obtain:
$$d_{n-1} = - \frac{Q_{10}}{P_{00}} = \frac{P_{11}}{Q_{01}} > 0,$$
that implies $P_{11}Q_{00} + Q_{10}Q_{01} = 0$. (Note that $P_{00} = Q_{01} = 0$ implies $Q_{10} = P_{11} = 0$ and do not satisfy the last inequality.) Multiplying the equalities $P_{00}d_{n-1} = - Q_{10}$ and $Q_{01}d_{n-1} = P_{11}$ by $P_{00}$ and $Q_{01}$ respectively and adding, we obtain $d_{n-1}(P^2_{00} + Q^2_{01}) = P_{11}Q_{01} - P_{00}Q_{10} > 0$.

If $P_{00} = Q_{01} = 0$ then, by the first equalities, $Q_{10} = P_{11} = 0$ that contradicts $Q_{01}^2 + P^2_{11} \neq 0$.  
$\square$
\begin{corollary}\label{Q0}
  System \eqref{2} admits no positive solution if and only if the following conditions hold: for $P_{00} \neq 0$
\begin{enumerate}
    \item[\rm 1.] $P_{00}P_{11} + Q_{10}Q_{01}  \neq 0$;
    \item[\rm 2.] $P_{11}Q_{01} \leq 0$ and $P^2_{11} + Q^2_{01} \neq 0$;
    \item[\rm 3.] $P_{11}Q_{01} > 0$ and $(P_{10}Q_{01} - P_{11}Q_{00})F(11,01) \geq 0$.
\end{enumerate}
For $P_{00} = 0$, either $P_{11}Q_{01} \leq 0$ or $$ P_{11}Q_{01} > 0 \qquad \mbox{and} \qquad (P_{10}Q_{01} - P_{11}Q_{00})F(11,01)\geq 0. $$
\end{corollary}
{\bf Proof.} For the proof, it is enough to show that in the case $P_{00}P_{11} + Q_{10}Q_{01} = 0$, the condition $P_{11}Q_{01} \leq 0$ contradicts the conditions of Lemma \ref{Q0}, namely, the inequality $P_{11}Q_{01} - P_{00}Q_{10} > 0$. Indeed, for $P_{11}Q_{01} - P_{00}Q_{10} > 0$, we should have $P_{11}Q_{01} > P_{00}Q_{10}$ and $P_{00}Q_{10} < 0$. From $P_{00}P_{11} + Q_{10}Q_{01}  = 0$ we obtain $Q_{10}Q_{01} = -P_{00}P_{11}$. Now $P_{11}Q_{01}P_{00}Q_{10} = - (P_{00}P_{11})^2 \leq 0$. Hence $P_{11}Q_{01}$ and  $P_{00}Q_{10}$ are of different signs, excluding the case $P_{11} = 0$ which implies $Q_{01} = 0$. $\square$

\subsection{Degenerate case} Here, we consider the conditions, that are necessary and sufficient for all the positive solutions $d^0_{1}, \ \ldots, \ d^0_{n-1}$ of the equation $F_0=0$ to satisfy $F_1(d^0_{1}, \ \ldots, d^0_{n-1})>0$.
\begin{theorem}\label{deg} The system 
$$ \begin{cases}F_0=0; \\
F_1 = 0 \end{cases}$$ have no common positive solution $(d^0_{1}, \ \ldots, d^0_{n-1})$, $d^0_i > 0$ if and only if one of the following cases holds.
\begin{enumerate}
    \item[\rm 1.] Either $F_{00} > 0$ and the equation
        $$F(00, 10) = 0, $$
        does not have any positive solutions $(d^0_{1}, \ \ldots, \ d^0_{n-2})$, for which $$G(00, 10) < 0.$$
        or $F_{00} = 0$ and $F_{10} > 0$.
    \item[\rm 2.] Either $F_{01} > 0$ and the equation
        $$F(01, 11) = 0$$
        does not have any positive solutions, for which $$G(01, 11) < 0.$$
        or $F_{01} = 0$ and $F_{11} > 0$.
    \item[\rm 3.] $\begin{cases}F(00, 10) = F(01, 11) = 0; \\ G(00, 10) < 0; \\ G(01, 11) < 0 \end{cases}$ and $$G(01,11)F_{00} - G(00,10)F_{01} \neq 0.$$
    For $Q_{01}Q_{00} \neq 0$, the last condition is equivalent to $$P_{10}Q_{01} - P_{11}Q_{00} \neq 0.$$
\end{enumerate}
\end{theorem}
{\bf Proof.} Let $F_0 = 0$. By Formula \eqref{JS},
$$ F_{0} = d_{n-1}^2F_{00} + 2d_{n-1}G(00,10) + F_{10}.$$ 
Since $F_{s} = P^2_{s} + Q^2_{s} \geq 0$ for every binary string $s$ and all variables $d_i \in {\mathbb R}$, we obtain that both $F_{00}, F_{10} \geq 0$. Let $F_{00} > 0$. Hence the equation $F_0 = 0$ is solvable in variable $d_{n-1}$ and the solution is positive if and only if the following conditions hold:
$$D = 4G^2(00,10) - 4F_{00}F_{10} \geq 0 \qquad \mbox{and} \qquad G(00,10) < 0.$$
Since $$D = 4((P_{00}Q_{10} - Q_{00}P_{10})^2 - (P_{00}^2 +Q_{00}^2)(P_{10}^2 + Q_{10}^2)) = $$ $$ -4(P_{00}P_{10} + Q_{00}Q_{10})^2 = - 4F^2(00,10),$$
we obtain $D > 0$ is impossible and $D = 0$ is possible if and only if $F(00,10) = 0$. Hence for $F_{00} \neq 0$ we obtain the conditions:
$$\begin{cases}F(00, 10) = 0; \\ G(00,10) < 0. \end{cases}$$
The unique solution in this case is $$d_{n-1} = -\frac{G(00,10)}{F_{00}}.$$
The condition $F_{00} = 0$ immediately implies $G(00,10) = 0$. Then the equation $F_0 = 0$ is solvable if and only if $F_{10} = 0$, and any $d_{n-1} > 0$ will satisfy it. 

Applying Formula \eqref{JS} to $F_1$, we obtain
$$ F_{1} = d_{n-1}^2F_{01} + 2d_{n-1}G(01,11) + F_{11}.$$ 
Let $F_{01} \neq 0$. The same reasoning shows that the equation $F_1 = 0$ is solvable if and only if $F(01,11) = 0$, with the unique solution $$d_{n-1} = -\frac{G(01,11)}{F_{01}},$$
which is positive if and only if $G(01,11)< 0$.
For all other values of $d_{n-1}$, $F_1 > 0$.

If $F_{01} = 0$, then $F_1 \equiv F_{11}$, the equation $F_1 = 0$ is solvable if and only if $F_{11} = 0$ and any $d_{n-1} > 0$ will satisfy it. 

Hence, for the simultaneous solvability of the system $$\begin{cases} F_0 = 0 \\ F_1 = 0 \end{cases}$$ of quadratic equations in $d_{n-1}$ we obtain the following cases:
\begin{enumerate}
\item[\rm 1.] $1$-degenerate case ($F_{00} \neq 0$, $F_{01} \neq 0$, both the equations $F_0 = 0$ and $F_1 = 0$ are solvable in positive values and their solutions coincide):
$$\begin{cases}F(01,11) = F(00,10) = 0; \\ G(00,10) < 0; \\ G(01,11) < 0;\\
-\frac{G(00,10)}{F_{00}} = -\frac{G(01,11)}{F_{01}}
\end{cases}$$
Considering the last equality, we denote:
$$E = G(01,11)F_{00} - G(00,10)F_{01}$$
Expressing $Q_{11} = -\frac{P_{11}P_{01}}{Q_{01}}$ (assuming $Q_{01} \neq 0$) from the condition $F(01,11) = 0$ and $Q_{10} = -\frac{P_{00}P_{10}}{Q_{00}}$ (assuming $Q_{00} \neq 0$) from the condition $F(00, 10) = 0$, and substituting, we get:
$$G(01,11) = P_{01}Q_{11} - Q_{01}P_{11} = -P_{01}\frac{P_{11}P_{01}}{Q_{01}}- Q_{01}P_{11} = $$ $$ -\frac{P_{11}}{Q_{01}}(P_{01}^2 + Q^2_{01}) = -\frac{P_{11}}{Q_{01}}F_{01}.$$
$$ G(00,10) = P_{00}Q_{10} - Q_{00}P_{10} = -P_{00}\frac{P_{00}P_{10}}{Q_{00}}- Q_{00}P_{10} = $$ $$ -\frac{P_{10}}{Q_{00}}(P_{00}^2 + Q^2_{00}) = -\frac{P_{10}}{Q_{00}}F_{00}.$$
Then substituting into $E$, we obtain:
$$E = F_{00}F_{01}\left(-\frac{P_{11}}{Q_{01}} + \frac{P_{10}}{Q_{00}}\right).$$
\item[\rm 2.] One of the equations $F_0 = 0$ or $F_1 = 0$ has a unique solution while the other is degenerate. This case gives one of the following systems:
$$\mbox{either} \ \begin{cases}F(00,10) = 0; \\ G(00,10) < 0; \\ F_{01} = F_{11} = 0;\\
\end{cases} \qquad \mbox{or} \qquad \begin{cases}F(01,11) = 0; \\ G(01,11) < 0; \\ F_{00} = F_{10} = 0;\\
\end{cases}$$
Then either $d_{n-1} = -\frac{G(00,10)}{F_{00}}$ or $d_{n-1} = -\frac{G(01,11)}{F_{01}}$.
\item[\rm 3.] $2$-degenerate case. Both the equations $F_0$ and $F_1$ are degenerate. Then $$F_{00} = F_{01} = F_{10} = F_{11} = 0$$ and the solution $d_{n-1}$ is any real value.
\end{enumerate}
$\square$

Note that the above construction also gives an alternative proof of Theorem \ref{Step1} through the analysis of the quadratic expression \eqref{Jrec}. 

\subsection{Step 2 Tree Partitioning}
Consider the following diagramm of equivalent conditions of D-stability.

\vspace{0.3cm}

\begin{center}
\begin{tikzpicture}[
    box/.style={
        draw, 
        rectangle, 
        rounded corners=1pt, 
        align=left, 
        minimum width=4.3cm, 
        inner sep=9pt,
        font=\small,
        thick
    },
    titlefont/.style={
        font=\small\bfseries\itshape,
        anchor=south west,
        inner sep=2pt,
        align=left
    },
    line/.style={
        thick
    }
]

\node[box, align=center] (L0) {
    $F(d_1, \dots, d_n) = 0$ \\
    has no positive solutions
};
\node[titlefont, anchor=south] at (L0.north) {Level 0. Johnson criterion (Theorem \ref{Johhn})};

\node[box, below left=2.3cm and 0.9cm of L0] (L1_nondeg) {
    $\begin{cases} F(0;1)(d_1, \dots, d_{n-1}) = 0 \\ G(0;1)(d_1, \dots, d_{n-1}) < 0 \end{cases}$ \\[0.5em]
    has no positive solutions
};
\node[titlefont] at (L1_nondeg.north west) {Level 1. Non-degenerate\\ case of Theorem \ref{Step1}};

\node[box, below right=2.3cm and 0.9cm of L0] (L1_deg) {
    $\begin{cases} F_0(d_1, \dots, d_{n-1}) = 0 \\ F_1(d_1, \dots, d_{n-1}) = 0 \end{cases}$ \\[0.5em]
    has no positive solutions
};
\node[titlefont] at (L1_deg.north west) {Level 1. Degenerate\\ case of Theorem \ref{Step1}};

\node[box, below=2.0cm of L1_nondeg, xshift=-1.4cm] (L2_left1) {
    $Q_{01}P_{00} > 0$ \\
    and the conditions \\
    of either $1a$, $1b$ \\
    or $1c$ holds
};
\node[titlefont] at (L2_left1.north west) {Level 2.};

\node[box, below=0.5cm of L2_left1] (L2_left2) {
    $Q_{01}P_{00} < 0$ \\
    and the conditions \\
    of either $2a$, $2b$ \\
    or $2c$ holds
};

\node[box, below=0.5cm of L2_left2] (L2_left3) {
    Either $Q_{01} = 0$ or $P_{00} = 0$ \\
    (but not both) \\
    and either $3a$ or $3b$ \\
    holds
};

\node[box, right=0.6cm of L2_left1] (L2_mid) {
    $Q_0 = 0$ and \\
    the conditions of \\
    Corollary \ref{Q0} hold
};
\node[titlefont] at (L2_mid.north west) {Level 2.};

\node[box, at={(L1_deg |- L2_left1)}] (L2_right) {
    One of the \\
    cases $1$, $2$ or $3$ \\
    of Theorem \ref{deg} \\
    holds
};
\node[titlefont] at (L2_right.north west) {Level 2.};

\draw[line] (L0.south) --++ (0,-0.5) -| (L1_nondeg.north);
\draw[line] (L0.south) --++ (0,-0.5) -| (L1_deg.north);

\coordinate (left_backbone) at ($(L2_left1.west) + (-0.4, 0)$);
\draw[line] (L1_nondeg.south) --++ (0,-0.5) -| (left_backbone) -- ($(L2_left3.west) + (-0.4, 0)$);
\draw[line] (left_backbone) -- (L2_left1.west);
\draw[line] ($(L2_left2.west) + (-0.4, 0)$) -- (L2_left2.west);
\draw[line] ($(L2_left3.west) + (-0.4, 0)$) -- (L2_left3.west);

\draw[line] (L1_nondeg.south) --++ (0,-0.5) -| (L2_mid.north);

\draw[line] (L1_deg.south) -- (L2_right.north);

\end{tikzpicture}
\end{center}

The flowchart tracks the execution of the Delete/Zero algorithm to verify the $D$-stability of an $n \times n$ matrix $\mathbf A$ by breaking down the multi-variable polynomial test $\det({\mathbf A}+i{\mathbf D}) \neq 0$. Starting from Johnson's top-level criteria, the algorithm sequentially expands the determinant using the recurrence relations of Lemma \ref{rec}. At Step 1, by eliminating the variable $d_n$, we decompose Johnson's criterion into two conditions on polynomials of $n-1$ positive variables $d_1, \ \ldots, \ d_{n-1}$. At Step 2 (Section 5), by strategically eliminating the variable $d_{n-1}$, we decompose the  conditions from Theorem \ref{Step1} into a structured set of conditions over $n-2$ variables. This algebraic reduction is designed to keep the polynomial degrees as low as possible (at most 
$4$). For non-degenerate matrices, it reduces the problem to checking whether the resulting quadratic equation $F(0,1)=0$ has real roots within the restricted domain $S$ where $G(0,1)<0$. For degenerate configurations, it applies the conditions of Theorem \ref{deg}.

Summarizing the preceding subsections, we obtain the following set of conditions.
\begin{theorem}[Level-2 equivalent conditions]
Let $\mathbf{A} \in \mathcal{M}^{n \times n}$ be stable. Then $\mathbf A$ is $D$-stable if and only if the following five conditions on the polynomials of $n-2$ variables hold.
\begin{enumerate}
    \item[\rm 1.]  For all the values $d_1, \ \ldots, \ d_{n-2}$ that satisfy $Q_{01}P_{00} > 0$, the conditions of either 1a, 1b or 1c hold;
    \item[\rm 2.] For all the values $d_1, \ \ldots, \ d_{n-2}$ that satisfy $Q_{01}P_{00} < 0$, the conditions of either 2a, 2b or 2c hold;
    \item[\rm 3.] For all the values $d_1, \ \ldots, \ d_{n-2}$ that satisfy $Q_{01}P_{00} = 0$, $Q^2_{01} + P^2_{00} \neq 0$, the conditions of either 3a or 3b hold;
    \item[\rm 4.] The conditions of Corollary \ref{Q0} hold;
    \item[\rm 5.] One of the conditions 1, 2 or 3 of Theorem \ref{deg} hold.
\end{enumerate}
\end{theorem}
{\bf Proof.} The proof follows from the above reasoning. $\square$

\section{Hierarchy of conditions sufficient for $D$-stability}
\subsection{Recurrence relations for $F(s,t)$ and $G(s,t)$} 
Consider the following result that recovers recurrence relations. 

\begin{lemma}\label{rec1} For any length $k$, $1 \leq k \leq n-1$ and any two binary strings $s = \{s_{n-k+1}, \ \ldots, \ s_n\}$ and $t = \{t_{n-k+1}, \ \ldots, \ t_n\}$, where $s_i, \ t_i \in \{0,1\},$
 the following recurrence relations hold for every set of real variables $(d_1, \ \ldots, \ d_{n-k})$:
\begin{equation}\label{expsystF}F(s,t) = F(0s,0t)d^2_{n-k} + (G(0s,1t) - G(1s,0t)) d_{n-k} + F(1s,1t); \end{equation}
\begin{equation}\label{expsystG}G(s,t) = G(0s,0t)d^2_{n-k} + ( F(1s,0t) - F(0s,1t)) d_{n-k} + G(1s,1t).\end{equation}
\end{lemma}
{\bf Proof.} For the proof, we use induction on the length of the binary strings $k$. 

Let $k = 1$. Then both of the strings contains just one element $s = s_n$ and $t = t_n$ and we have four options: $(s,t) = (0,1), \  (1,0), \ (0,0)$ or $(1, 1)$. 

Consider the first two options: $s=0$, $t = 1$ and $s = 1$, $t = 0$. Note that $F(0,1) = F(1,0)$ and $G(0,1) = - G(1,0)$. Therefore Formulas \eqref{expsystF} and \eqref{expsystG} are already proven (see \eqref{Fexp} and \eqref{Gexp}): 
$$ F(0,1) = F(00,01)d^2_{n-1} + (G(00,11) - G(10,01)) d_{n-1} + F(10,11);$$
$$G(0,1) = G(00,01)d^2_{n-1} + ( F(01,10) - F(00,11)) d_{n-1} + G(10,11).$$
Now consider the second two options: $s=t=0$ and $s=t=1$. In this case, we have by \eqref{JS}:
$$F(s,s) = d_{n-1}^2F(0s,0s) + 2d_{n-1}G(0s,1s) + F(1s,1s), \qquad s = 0, \ 1.$$
Hence Formulas \eqref{expsystF} and \eqref{expsystG} hold.

Now let Formulas \eqref{expsystF} and \eqref{expsystG} hold for any binary strings $t$ and $s$ of length $n-k$. Prove them for the length $n-k+1$. 
By Lemma \ref{rec},
$$\begin{cases}P_{s} = -d_{n-k}Q_{0s} + P_{1s}; \\
Q_{s} = d_{n-k}P_{0s} + Q_{1s}.
\end{cases} \qquad \begin{cases}P_{t} = -d_{n-k}Q_{0t} + P_{1t}; \\
Q_{t} = d_{n-k}P_{0t} + Q_{1t}.
\end{cases}$$ 
Substituting into Formula \eqref{expsystF} and collecting with respect to the powers of $d_{n-1}$, we obtain for $F(s,t)$ the following quadratic expression in variable $d_{n-1}$:
$$F(s,t) = P_sP_t + Q_sQ_t = (P_{0s}P_{0t} + Q_{0s}Q_{0t})d_{n-1}^2 + $$
$$(P_{0s}Q_{1t} - Q_{0s}P_{1t} + P_{0t}Q_{1s} - Q_{0t}P_{1s})d_{n-1} + P_{1s}P_{1t} +Q_{1s}Q_{1t}.$$
Using the notations \eqref{FGs}, namely
$$F(0s,0t) = P_{0s}P_{0t} + Q_{0s}Q_{0t}; \qquad F(1s,1t) = P_{1s}P_{1t} +Q_{1s}Q_{1t}; $$
$$G(0s, 1t) = P_{0s}Q_{1t} - Q_{0s}P_{1t}; \qquad G(1s,0t)= P_{1s}Q_{0t} - Q_{1s}P_{0t} $$
we obtain exactly Formula \eqref{expsystF}:
$$ F(s,t) = F(0s,0t)d^2_{n-1} + (G(0s,1t) - G(1s,0t)) d_{n-1} + F(1s,1t).$$
Analogically, for $G(s,t)$, we obtain:
$$G(s,t) = (P_{0s}Q_{0t} - Q_{0s}P_{0t})d^2_{n-1} +$$
$$(P_{1s}P_{0t} + Q_{1s}Q_{0t} -(P_{0s}P_{1t} + Q_{0s}Q_{1t}))d_{n-1} + P_{1s}Q_{1t} - Q_{1s}P_{1t}.$$
Since
$$G(0s,0t) = P_{0s}Q_{0t} - Q_{0s}P_{0t}; \qquad G(1s,1t) = P_{1s}Q_{1t} - Q_{1s}P_{1t}; $$
$$F(0s,1t) = P_{0s}P_{1t} + Q_{0s}Q_{1t}; \qquad F(1s,0t)= P_{1s}P_{0t} + Q_{1s}Q_{0t} $$
we obtain:
$$G(s,t) = G(0s,0t)d^2_{n-1} + ( F(1s,0t) - F(0s,1t)) d_{n-1} + G(1s,1t),
$$which is exactly \eqref{expsystF}. $\square$

Lemma \ref{rec1} implies that each $F(s,t)$ and $G(s,t)$ are polynomials of degree at most $2$ on each remaining variable $d_1, \ \ldots, \ d_{n-k}$.

\subsection{Branched tests for $D$-stability}
Consider the following branched process, starting from $F(0,1)$.

{\bf Process I.} Denote $c_{\emptyset}:=F(0,1)$.

{\bf Step 1.} Applying \eqref{Fexp}, we collect $F(0,1)$ according to the powers of $d_{n-1}$ and extract the coefficients: $$F(0,1) = F(00,01)d^2_{n-1} + (G(00,11) - G(10,01)) d_{n-1} + F(10,11) \mapsto (c_1, \ c_2, \ c_3),$$
where $c_1 :=  F(00,01)$, $c_2 := G(00,11) - G(10,01)$, $c_3 := F(10,11)$. 

{\bf Step 2.} Applying Lemma \ref{rec1} to each of the branches $c_i$, $i = 1, \ 2, \ 3$, collecting with respect to the powers of $d_{n-2}$ and extracting the coefficients, we get:
$$c_1 = F(00,01) = F(000,001)d^2_{n-2} + $$$$(G(000,101) - G(100,001)) d_{n-2} + F(100,101) \mapsto (c_{11}, \ c_{21}, \ c_{31}),$$
where $c_{11} =  F(000,001)$, $c_{21} = G(000,101) - G(100,001)$, $c_{31} = F(100,101)$. 
$$c_2 = G(00,11) - G(10,01) = G(000,011)d^2_{n-2} + (F(100,011) - F(000,111)) d_{n-2} +$$ $$ G(100,111) - G(010,001)d^2_{n-2} - ( F(110,001) - F(010,101)) d_{n-2} - G(110,101) \mapsto $$
$$(c_{12}, \ c_{22}, \ c_{32}),$$
where $c_{12} =  G(000,011)- G(010,001)$, $c_{22} = F(100,011) - F(000,111)- F(110,001) + F(010,101)$, $c_{32} = G(100,111) - G(110,101)$.
$$c_3 = F(10,11) = F(010,011)d^2_{n-2} + $$$$(G(010,111) - G(110,011)) d_{n-2} + F(110,111) \mapsto (c_{13}, \ c_{23}, \ c_{33}),$$
where $c_{13} =  F(010,011)$, $c_{23} = G(010,111) - G(110,011)$, $c_{33} = F(110,111)$. 

{\bf Step $i$.} After Step $i-1$, we have $3^{i-1}$ polynomials $c_{k_{i-1}\ldots k_{1}}$, where $k_j \in \{1,2,3\}$. Applying Lemma \ref{rec1} to each of the branches $c_{k_{i-1}\ldots k_{1}}$, collecting with respect to the powers of $d_{n-i}$ and extracting the coefficients, we obtain  $$c_{k_{i-1}\ldots k_{1}} \mapsto (c_{1k_{i-1}\ldots k_{1}}, c_{2k_{i-1}\ldots k_{1}}, c_{3k_{i-1}\ldots k_{1}}).$$

Algorithm stops at $i = n-1$, giving $3^{n-1}$ determinantal expressions.

By analogy, the branched process, starting from $G(0,1)$ is as follows.

{\bf Process II.} $\hat{c}_{\emptyset}:=G(0,1)$.

{\bf Step 1.} Applying \eqref{Gexp}, we get: $$G(0,1) = G(00,01)d^2_{n-1} + (F(10,01) - F(00,11)) d_{n-1} + G(10,11) \mapsto (\hat{c}_1, \ \hat{c}_2, \ \hat{c}_3),$$
where $\hat{c}_1 :=  G(00,01)$, $\hat{c}_2 := F(10,01) - F(00,11)$, $\hat{c}_3 := G(10,11)$. 

{\bf Step 2.} Applying Lemma \ref{rec1} to each of the branches $\hat{c}_i$, $i = 1, \ 2, \ 3$, collecting with respect to the powers of $d_{n-2}$ and extracting the coefficients, we get:
$$\hat{c}_1 = G(00,01) = G(000,001)d^2_{n-2} + $$$$(F(100,001) - F(000,101)) d_{n-2} + G(100,101) \mapsto (\hat{c}_{11}, \ \hat{c}_{21}, \ \hat{c}_{31}),$$
where $\hat{c}_{11} =  G(000,001)$, $\hat{c}_{21} = F(100,001) - F(000,101)$, $\hat{c}_{31} = G(100,101)$. 
$$\hat{c}_2 = F(10,01) - F(00,11) =$$ $$ F(010,001)d^2_{n-2} + (G(010,101) - G(110,001) ) d_{n-2} + F(110,101) - $$ $$F(000,011)d^2_{n-2} - (G(000,111) - G(100,011)) d_{n-2} - F(100,111) \mapsto $$
$$(\hat{c}_{12}, \ \hat{c}_{22}, \ \hat{c}_{32}),$$
where $\hat{c}_{12} = F(010,001) -  F(000,011)$, $\hat{c}_{22} = G(010,101) - G(110,001) - G(000,111) + G(100,011)$, $\hat{c}_{32} =  F(110,101) - F(100,111)$.
$$\hat{c}_3 = G(10,11) = G(010,011)d^2_{n-2} + $$$$(F(110,011) - F(010,111)) d_{n-2} + G(110,111) \mapsto (\hat{c}_{13}, \ \hat{c}_{23}, \ \hat{c}_{33}),$$
where $\hat{c}_{13} =  G(010,011)$, $\hat{c}_{23} = F(110,011) - F(010,111)$, $\hat{c}_{33} = G(110,111)$. 

{\bf Step $i$.} Having $3^{i-1}$ polynomials $\hat{c}_{k_{i-1}\ldots k_{1}}$, where $k_j \in \{1,2,3\}$, we apply Lemma \ref{rec1} to each of the branches $\hat{c}_{k_{i-1}\ldots k_{1}}$. Collecting with respect to the powers of $d_{n-i}$ and extracting the coefficients, we obtain  $$\hat{c}_{k_{i-1}\ldots k_{1}} \mapsto (\hat{c}_{1k_{i-1}\ldots k_{1}}, \hat{c}_{2k_{i-1}\ldots k_{1}}, \hat{c}_{3k_{i-1}\ldots k_{1}}).$$

As the previous one, this algorithm stops at $i = n-1$, giving $3^{n-1}$ determinantal expressions.

\begin{theorem}\label{test} Let ${\mathbf A} \in {\mathcal M}^{n\times n}$ be stable. Then for $\mathbf A$ be $D$-stable, it is sufficient that for some $i$, $1 \leq i \leq n-1$, at least one of Test I or Test II holds. 
\begin{enumerate}
    \item[\rm Test I.] All $3^i$ polynomial coefficients $c_{k_i\ldots k_1}$, $k_j \in \{1, \ 2, \ 3\}$, obtained by applying Process I for $i$ times satisfy
    $$c_{k_i\ldots k_1}(d_1, \ldots, d_{n-i-1} ) \geq 0$$  
    for every set of positive variables $d_1, \ \ldots, \ d_{n-i-1}$, with at least one inequality being strict.
    \item[\rm Test II.] All $3^i$ polynomial coefficients $\hat{c}_{k_i\ldots k_1}$, $k_j \in \{1, \ 2, \ 3\}$, obtained by applying Process II for $i$ times satisfy
    $$\hat{c}_{k_i\ldots k_1}(d_1, \ldots, d_{n-i-1} ) \geq 0$$  
    for every set of positive variables $d_1, \ \ldots, \ d_{n-i-1}$, with at least one inequality being strict.
\end{enumerate}
\end{theorem}
{\bf Proof.} It is enough to prove Test I (Test II can be proven analogically). By recurrent relations from Process I, 
$$c_{k_{i-1}\ldots k_1} = c_{1k_{i-1}\ldots k_1}d^2_{n-i} + c_{2k_{i-1}\ldots k_1}d_{n-i} + c_{3k_{i-1}\ldots k_1}.$$
Since $$c_{k_ik_{i-1}\ldots k_1}(d_1, \ldots, d_{n-i-1} ) \geq 0, \qquad k_i = 1, \ 2, \ 3$$  
    for every set of positive variables $d_1, \ \ldots, \ d_{n-i-1}$, we have $$c_{k_{i-1}\ldots k_i}(d_1, \ldots, d_{n-i-1},d_{n-i} ) \geq 0$$  for every set of positive variables $d_1, \ \ldots, \ d_{n-i}$. Moreover, if $$c_{k^0_i\ldots k^0_1}(d_1, \ldots, d_{n-i-1} ) > 0$$ for at least one set of indices $k^0_i\ldots k^0_1$, then at least one coefficient in the quadratic expression for $c_{k_{i-1}\ldots k_1}$ is positive 
and $$c_{k^0_{i-1}\ldots k^0_1}(d_1, \ldots, d_{n-i} )= c_{1k^0_{i-1}\ldots k^0_1}d^2_{n-i} + c_{2k^0_{i-1}\ldots k^0_1}d_{n-i} + c_{3k^0_{i-1}\ldots k^0_1} > 0.$$
Applying this reasoning for $i-1$ times, we obtain
$$F(0,1) = c_{\emptyset} = c_1d^2_{n-1} + c_2d_{n-1} + c_3 > 0$$
for all positive variables $d_1, \ \ldots, \ d_{n-1}$,
since $c_j \geq 0$ for $j = 1, \ 2, \ 3,$ with at least one inequality being strict. Applying Corollary \ref{suff}, we complete the proof.
 $\square$

By stopping the recursion at depth $i$, we obtain tests of increasing power and complexity, tunable to the user's computational resources. Moreover, we may stop each branch at any level, continuing the process for the other branches. The recursive framework, proposed by Theorem \ref{test} decomposes $D$-stability into low-degree multivariate polynomial leaf conditions, checkable via semidefinite programming (see, for example, \cite{BLE}, \cite{LA}). This can offer relaxation when checking sparse or structured matrices (i.e., matrices with symmetries). Note, that $D$-stability tests \cite{PAV1}, \cite{PAV2} also propose symbolic-numeric approach to $D$-stability characterization, optimizing the verification of Johnson's conditions (Theorem \ref{stabcond}). The two approaches are complementary. Pavani's symbolic technique could be potentially applied through the recursive framework to verify the conditions on each branch.

A natural question arises regarding the conservative nature of enforcing the nonnegativity of the intermediate polynomial coefficients, i.e. requiring $c_{2k_{i-1}\ldots k_1} \geq 0$ in a generic quadratic branch $$c_{k_{i-1}\ldots k_1} = c_{1k_{i-1}\ldots k_1}d^2_{n-i} + c_{2k_{i-1}\ldots k_1}d_{n-i} + c_{3k_{i-1}\ldots k_1} \geq 0.$$ Mathematically, the exact condition for positivity over the positive orthant allows for $c_{2k_{i-1}\ldots k_1} < 0$ provided that the discriminant satisfies $$c^2_{2k_{i-1}\ldots k_1} - 4c_{1k_{i-1}\ldots k_1}c_{3k_{i-1}\ldots k_1} < 0.$$
However, the coefficients $c_{k_ik_{i-1}\ldots k_1}$, $k_i = 1, \ 2, \ 3$ are themselves polynomials in the remaining diagonal scaling variables $d_1, \dots, d_{n-i-1}$. Evaluating the discriminant instantly doubles the algebraic degree of these parameters from degree $2$ (multi-quadratic) to degree $4$ (multi-quartic). A degree-$4$ explosion completely destroys the structural invariance required to propagate the recursion down the ternary tree.  

Denote ${\mathcal M}_i$ $(0 \leq i \leq n-1)$ the subset of stable $n \times n$ matrices, for which all $3^i$ coefficients $c_{k_i, \ldots, k_1}$, $k_j \in \{1, \ 2, \ 3\}$, $j = 1, \ \ldots, \ i$ are positive. Hence ${\mathcal M}_0$ is the subset of stable matrices satisfying $F(0,1) > 0$ for all positive values $d_1, \ \ldots, \ d_{n-1}$ and ${\mathcal M}_{n-1}$ is the subset of stable matrices, determined by final $3^{n-1}$ inequalities for its principal minors. This yields to the following inclusive relations:
$${\mathcal M}_{n-1} \subset {\mathcal M}_{n-2} \subset \ldots \subset {\mathcal M}_1 \subset {\mathcal M}_0 \subset \{\mbox{$D$-stable matrices}\}.$$
Obviously, as we descend the recursion, we impose more restrictive positivity requirements, thereby certifying fewer stable matrices as $D$-stable. The deepest class ${\mathcal M}_{n-1}$ is the most conservative; the shallowest ${\mathcal M}_0$ is the most inclusive. In practice, this theoretical hierarchy may be altered by the conservative nature of SOS methods. Because verifying the positivity of multivariable polynomials over the positive orthant is NP-hard, semidefinite relaxations may fail to identify the full volume of ${\mathcal M}_0$. As a result of these numerical relaxation gaps, the empirical inclusion domain may expand at the middle levels, making the deeper, lower-variable tree branches (${\mathcal M}_{i}$) paradoxically more dependable to verify numerically, albeit at the cost of increased structural conservatism.

\section{Numerical experiments and examples}
\subsection{Example of early stopping}
Let $${\mathbf A} = \begin{pmatrix}2 & -2 & 1 & 0 & 0 \\ 1 & 0 & 0 & 0 & -1 \\ 1 & -1 & 1 & 0 & 0 \\ 0 & -1 & 0 & 1 & -1 \\ 0& 1 & 0 & 0 & 2 \\\end{pmatrix}.$$
It is known to be $D$-stable (see \cite{OLP}, p. 423). We confirm this result by using Tests I, stopping the recursion at level $2$.
\begin{enumerate}
    \item[\rm Step 0.] For parameters $\{d_1, \ d_2, \ d_3, \ d_4\}$, we calculate symbolic polynomial $$ F(0,1) = 3 + 3 d_4^2 - 4 d_1 d_2 - 4 d_4^2 d_1 d_2 + 2 d_2^2 + 2 d_4^2 d_2^2 + 2 d_1^2 d_2^2 + 
 2 d_4^2 d_1^2 d_2^2 + $$$$ d_1 d_3 + d_4^2 d_1 d_3 + 4 d_1 d_2^2 d_3 + 4 d_4^2 d_1 d_2^2 d_3 + 12 d_3^2 + 
 12 d_4^2 d_3^2 - 8 d_1 d_2 d_3^2 - 8 d_4^2 d_1 d_2 d_3^2 + $$$$ 8 d_2^2 d_3^2 + 8 d_4^2 d_2^2 d_3^2 + 
 2 d_1^2 d_2^2 d_3^2 + 2 d_4^2 d_1^2 d_2^2 d_3^2.$$
      Some of the coefficients of the monomials in $F(0,1)$ are negative, hence Test I on level $n-1$ will fail.
   \item[\rm Step 1.] Collect according to the powers of $d_4$ and extract the coefficients $c_1, \ c_2, \ c_3$:
   $$c_1 = 3 - 4 d_1 d_2 + 2 d_2^2 + 2 d_1^2 d_2^2 + d_1 d_3 + 4 d_1 d_2^2 d_3 + 12 d_3^2 - 
    8 d_1 d_2 d_3^2 + 8 d_2^2 d_3^2 + 2 d_1^2 d_2^2 d_3^2; $$
    $$ c_2 = 0;$$
   $$c_3 = 3 - 4 d_1 d_2 + 2 d_2^2 + 2 d_1^2 d_2^2 + d_1 d_3 + 4 d_1 d_2^2 d_3 + 12 d_3^2 - 
 8 d_1 d_2 d_3^2 + 8 d_2^2 d_3^2 + 2 d_1^2 d_2^2 d_3^2.$$
 Note that $c_1 \equiv c_3$ for all $d_1, \ d_2, \ d_3$.
   \item[\rm Step 2.] Collect $c_1$ according to the powers of $d_3$ and extract $c_{11}$, $c_{21}$, $c_{31}$: $$c_{11} = 12 - 8 d_1 d_2 + 8 d_2^2 + 2 d_1^2 d_2^2;$$ $$c_{21} = d_1 + 4 d_1 d_2^2; $$ $$c_{31} = 3 - 4 d_1 d_2 + 2 d_2^2 + 
 2 d_1^2 d_2^2.$$ 
 Note, that $c_{21} > 0$ for all positive values $d_1, d_2$ and stop checking this branch.
 \item[\rm Step 3.] Consider the coefficients $c_{11}$ and $c_{31}$. Collecting them according to $d_2$, we obtain 
$$c_{11} = 3 - 4 d_1 d_2 + (2 + 2 d_1^2) d_2^2$$ and 
$$c_{31} = 12 - 8 d_1 d_2 + (8 + 2 d_1^2) d_2^2.$$ 
Solving the inequality $D< 0$ for both quadratic equations, we obtain $16 d_1^2 - 12 (2 + 2 d_1^2) < 0 $ and $64 d_1^2 - 48 (8 + 2 d_1^2) < 0 $ for all positive values $d_1$. Hence both $c_{11}$ and $c_{31}$ are positive for all positive values of $d_1$ and $d_2$. We conclude that $\mathbf A$ is $D$-stable.
\end{enumerate}
This example illustrates the key feature of our framework, namely, the ability to stop the recursion early. For a user with limited computational resources, testing the conditions at depth $i$ (involving $3^i$ polynomials in 
$n-i-1$ variables) may be feasible even when full recursion to depth $n-1$ is not. Developing an adaptive algorithm that automatically selects the optimal depth based on the matrix structure is a direction for future research.

\subsection{Population studies on structured matrices}
All numerical experiments and algorithmic verification routines were implemented in C++ and executed on a workstation equipped with an 11th Gen Intel(R) Core(TM) i7-11800H processor clocking at 2.30GHz with 16GB of RAM running a 64-bit operating system. 

To evaluate the empirical performance, conservatism, and scalability of our recursive framework, we construct test populations across two broad structural types: sign-symmetric $P$-matrices and diagonally stable matrices (both sufficient for $D$-stability, see \cite{JOHN1}, \cite{KU2}). For each matrix class, a population of of $N = 1 000$ random instances was generated for dimension $n = 6$ according to the precise algebraic protocols described below. 

{\bf Type 1: Sign-Symmetric $P$-Matrices.}
Recall, that a matrix ${\mathbf A} \in {\mathbb R}^{n \times n}$ is called a {\it sign-symmetric $P$-matrix}, if it satisfies the following conditions:
$$A\begin{pmatrix}\alpha \\ \alpha \end{pmatrix} > 0; \qquad A\begin{pmatrix}\alpha \\ \beta \end{pmatrix}A\begin{pmatrix}\beta \\ \alpha \end{pmatrix} \geq 0$$
 for all sets of indices $\alpha = (i_1, \ \ldots, \ i_k)$, $1 \leq i_1 < \ldots < i_k \leq n$ and $\beta = (j_1, \ \ldots, \ j_k)$, $1 \leq j_1 < \ldots < j_k \leq n$, and each order $k$, $1 \leq k \leq n$ (see \cite{CARL}, \cite{KU2}). We evaluate three classical subclasses (for the definitions and properties see, for example, \cite{GANT}):
\begin{enumerate}
\item[] {\bf Class 1a (Symmetric Positive Definite Matrices)}: to generate a random instance a dense matrix ${\mathbf B} = \{b_{ij}\}_{i,j = 1}^n$ is constructed with independent and identically distributed (i.i.d.) entries drawn from a continuous uniform distribution, $b_{ij} \sim \mathcal{U}(0,1)$. The symmetric matrix $\mathbf A$ is then formed by $${\mathbf A}: = {\mathbf B}{\mathbf B}^T + \epsilon{\mathbf I},$$ where a random shift variable $\epsilon \sim \mathcal{U}(0,1)$ is taken to ensure strict positive definiteness and prevent singularities.
\item[] {\bf Class 1b (Jacobi Oscillatory Matrices)}: By the classical Gantmacher--Krein factorisation theorem $\cite{GANT}$, a tridiagonal matrix is oscillatory if and only if it can be decomposed into a product of positive bidiagonal matrices. We construct these instances using the following algorithm:
\begin{enumerate}
\item[\rm 1.] Generate a random lower bidiagonal matrix $\mathbf{L} \in \mathcal{M}_{n \times n}$:
   $$\mathbf{L} := \begin{pmatrix} \alpha_1 & 0 & \dots & 0 & 0\\ \gamma_1 & \alpha_2 & \dots & 0 & 0 \\ 0 & \gamma_2 & \dots & 0 & 0 \\ \vdots & \vdots & \ddots & \vdots & \vdots \\ 0 & 0 &\dots & \gamma_{n-1} & \alpha_n \end{pmatrix},$$ 
   where the main diagonal elements $\alpha_i$ and lower subdiagonal elements $\gamma_i$ are drawn independently from a continuous uniform distribution, $\alpha_i, \gamma_i \sim \mathcal{U}(0.1, 10)$.
  \item[\rm 2.] Generate a random upper bidiagonal matrix $\mathbf{U} \in \mathcal{M}_{n \times n}$ with a unit main diagonal:
   $$\mathbf{U} := \begin{pmatrix} 1 & \beta_1 & 0 & \dots & 0\\ 0 & 1 & \beta_2 & \dots & 0 \\ \vdots & \vdots & \ddots & \ddots & \vdots \\ 0 & 0 & \dots & 1 & \beta_{n-1} \\ 0 & 0 &\dots & 0 & 1 \end{pmatrix},$$ 
   where the upper superdiagonal elements $\beta_i$ are drawn independently from $\beta_i \sim \mathcal{U}(0.1, 10)$.
   \item[\rm 3.] Compute the target tridiagonal oscillatory test matrix $\mathbf A$ via the product:
   $$\mathbf{A} := \mathbf{L}\mathbf{U}.$$ 
\end{enumerate}
\item[] {\bf Class 1c (Vandermonde Matrices)}: Vandermonde matrices generated from distinct positive nodes are well-known to be oscillatory matrices $\cite{GANT}$. The test instances are constructed as follows:
\begin{enumerate}
    \item[\rm 1.] Define an initial positive node $\alpha_0 \sim \mathcal{U}(0.1, 10.0)$. A strictly increasing sequence of positive, distinct parameters $\{\alpha_j\}_{j=0}^{n-1}$ is generated recursively by drawing the incremental node separations from a continuous uniform distribution:
   $$\alpha_j = \alpha_{j-1} + \delta_j, \quad \delta_j \sim \mathcal{U}(0.1, 10.0), \quad j = 1, \dots, n-1.$$ 
   \item[\rm 2.] Construct the test matrix
   $$\mathbf{A} := \{a_{ij}\}_{i,j=0}^{n-1}, \quad a_{ij} = \alpha_j^i.$$ 
\end{enumerate}   
\end{enumerate}
{\bf Type 2: Diagonally Stable Matrices.}
Recall that a matrix ${\mathbf A} \in {\mathcal M}^{n \times n}$ is called {\it diagonally (positive) stable} if the matrix
$$ {\mathbf W} = {\mathbf D}{\mathbf A} + {\mathbf A}^T{\mathbf D}$$
is positive definite for some positive diagonal matrix ${\mathbf D}$ (see \cite{CROSS}, \cite{KAB}). We evaluate three standard subclasses:
\begin{enumerate}
\item[] {\bf Class 2a (Taussky Matrices)}: Following Taussky $\cite{TAU1}$, these matrices take the structural form:
$$\mathbf{A} := \delta\mathbf{I} + \mathbf{B} - \mathbf{B}^T,$$ 
where $\mathbf{B}$ is a random dense matrix with entries $\mathbf{B}_{ij} \sim \mathcal{U}(0,1)$, and the shift parameter $\delta$ is drawn from $\mathcal{U}(0.1, 10)$ to enforce strict diagonal stable properties.
\item[] {\bf Class 2b ($M$-Matrices)}: Nonsingular $M$-matrices are generated using their spectral radius characterization:
$$\mathbf{A} := (\rho(\mathbf{B}) + \delta)\mathbf{I} - \mathbf{B},$$ 
where $\mathbf{B}$ is a non-negative matrix with entries $\mathbf{B}_{ij} \sim \mathcal{U}(0, 1)$, $\rho(\mathbf{B})$ denotes its spectral radius, computed numerically via a deterministic Power Iteration algorithm configured for $2 000$ iterations to ensure asymptotic convergence to the dominant eigenvalue, and $\delta \sim \mathcal{U}(0.1, 10)$ ensures a robust stability margin.
\item[] {\bf Class 2c (Strictly Diagonally Dominant Matrices)}: Let $\mathbf{B} = \{b_{ij}\}_{i,j = 1}^n$ be an arbitrary dense base matrix with entries $b_{ij} \sim \mathcal{U}(0, 1)$. The row-wise Gershgorin radii are explicitly computed as $r_i := \sum_{j \neq i} \vert{}b_{ij}\vert{}$. The diagonally dominant test matrix is then constructed via:
$$\mathbf{A} := \diag\{r_1 + \delta_1, \ \ldots, \ r_n + \delta_n\} + {\mathbf B},$$ 
where each perturbation parameter $\delta_i$ is drawn independently from $\mathcal{U}(0.1, 10)$, guaranteeing that $a_{ii} > r_i$ for all $i$.
\end{enumerate}
To evaluate these theoretical boundaries empirically, our numerical validation is structured into the following experiment executed across the generated matrix populations, which is compared to the authoritative LMI-based test of Oliveira and Peres \cite{OLP} (see also \cite{KU2}, p. 677, Theorem 60), implemented and realized through the MOSEK Optimization Suite.

{\bf Experiment 1 (Static Level Isolation)}: This protocol isolates specific operational depths by running six entirely independent, non-cascading verification processes across the entire global population: the first evaluated exclusively at Level 0, checking the  positivity of the $(n-1)$-variable polynomial $F(0,1)$, the $i$th $(i = 2, \ \ldots, \ n-1)$ evaluated at the level $i-1$, checking positivity of $3^{i-1}$ polynomial coefficients, and the last one evaluated at the deepest possible Level $n-1$, checking positivity of $3^{n-1}$ coefficients, that are just numbers. By decoupling these layers, Experiment 1 directly measures the trade-offs between the highly inclusive but computationally heavy global polynomial check-in, and the highly conservative numerical evaluation at deeper branches. 

\begin{table}[htbp]
\centering
\caption{Experiment 1. Type 1 (Sign-Symmetric $P$-Matrices), $n=6$}
\label{tab:class1_n6_diagram}
\small
\begin{tabular}{ccccc}
\toprule
\textbf{Class} & \textbf{Total ($N$)} & \textbf{Termination Level ($i$)} & \textbf{Test I Rate (\%)} & \textbf{Test II Rate (\%)} \\
\midrule
\multirow{6}{*}{1a. Symmetric PD} & \multirow{6}{*}{1,000} 
    & 0 & 100.0 & 100.0 \\
    & & 1 & 43.5  & 100.0 \\
    & & 2 & 26.7  & 100.0  \\
    & & 3 & 18.2  & 100.0 \\
    & & 4 & 14.9 & 100.0 \\
    & & 5 & 14.9  & 100.0 \\ \cmidrule{1-5}
\multirow{5}{*}{1b. Jacobi Oscillatory} & \multirow{5}{*}{1,000} 
    & 0 & 100.0  & 100.0  \\
    & & 1 & 100.0  & 100.0  \\
    & & 2 & 100.0  & 100.0  \\
    & & 3 & 100.0 & 100.0 \\
    & & 4 & 100.0 & 100.0 \\
     & & 5 & 100.0 & 100.0 \\ \cmidrule{1-5}
\multirow{5}{*}{1c. Vandermonde} & \multirow{5}{*}{1,000} 
    & 0 & 100.0 & 100.0  \\
    & & 1 & 100.0  & 100.0  \\
    & & 2 & 100.0  & 100.0  \\
    & & 3 & 100.0& 100.0\\
    & & 4 & 100.0  & 100.0  \\
    & & 5 & 100.0  & 100.0  \\
\bottomrule
\end{tabular}
\end{table}

\begin{table}[htbp]
\centering
\caption{Experiment 1. Type 2 (Diagonally Stable Matrices), $n=6$}
\label{tab:class2_n6_diagram}
\small
\begin{tabular}{ccccc}
\toprule
\textbf{Class} & \textbf{Total ($N$)} & \textbf{Termination Level ($i$)} & \textbf{Test I Rate (\%)} & \textbf{Test II Rate (\%)} \\
\midrule
\multirow{5}{*}{2a. Taussky} & \multirow{5}{*}{1,000} 
    & 0 & 100.0& 0.0  \\
    & & 1 & 0.0  & 0.0  \\
    & & 2 & 0.0  & 0.0  \\
    & & 3 & 0.0  & 0.0  \\
    & & 4 & 0.0  & 0.0  \\
     & & 5 & 0.0  & 0.0  \\ \cmidrule{1-5}
\multirow{5}{*}{2b. $M$-matrices} & \multirow{5}{*}{1,000} 
    & 0 & 100.0 & 100.0 \\
    & & 1 & 100.0  & 100.0 \\
    & & 2 & 100.0 & 12.4  \\
    & & 3 & 24.4 & 1.4  \\
    & & 4 & 6.4 & 0.4  \\
    & & 5 & 6.4  & 0.4  \\ \cmidrule{1-5}
\multirow{5}{*}{2c. SDD matrices} & \multirow{5}{*}{1,000} 
    & 0 & 100.0  & 26.7 \\
    & & 1 & 38.4  & 26.3 \\
    & & 2 & 20.3  & 2.0  \\
    & & 3 & 2.2 & 0.1  \\
    & & 4 & 0.2  & 0.0  \\
    & & 5 & 0.2  & 0.0  \\
\bottomrule
\end{tabular}
\end{table}

 \textbf{LMI Test:} All the generated $D$-stable matrices of both Type 1 and Type 2 happened to be recognized and verified as $D$-stable by the Oliveira--Peres LMI test.

\subsection{Identification of random unstructured matrices}
 We test for $D$-stability an amount of randomly generated stable matrices by both Test I and Test II, working on the deepest and most conservative level $n-1$. For $n = 5$, on average, Test I identifies one $D$-stable matrix out of every $1000$ stable matrices. Test II gives positive answer much more seldom (not a single matrix was found to satisfy both Test I and Test II). For $n=6$, Test I found one $D$-stable matrix out of several millions randomly generated stable matrices. For $n=7$, one million of stable matrices was checked for less than $10$ minutes and have found nothing. Hence we conclude that Test I, applied to $n=5$, yields a non-zero hit rate ($\approx 10^{-3}$), for $n=6$ a marginal detection rate ($\approx 10^{-7}$), and for $n=7$ exactly zero positives in $10^6$ trials. Both Test I and Test II are capable of processing millions of $7 \times 7$ matrices in minutes, however, they are rejecting true $D$-stable matrices due to algorithmic limitations. 
 
 The experiment confirms that $D$-stable matrices form a structurally rare subset of the set of all stable matrices. As the dimension $n$ increases, the conditions, necessary for $D$-stability become significantly more restrictive (see, for example, \cite{HART}). Thus, the observed decay in hit rates represents a joint phenomenon: it reflects both the intrinsic, dimension-driven scarcity of the $D$-stability variety and the maximum geometric conservatism induced by truncating the ternary branching tree at its deepest possible level.

Consider an example of $6 \times 6$ matrix, verified by Test I:
$$ {\mathbf A} : = \begin{pmatrix}
  100.00  &  -5.67  &   1.89  &   2.29  &   9.05 &  -38.42 \\ 
  -14.64  &  53.17  &  11.64   &  1.16  &  -8.78  &  46.73 \\ 
  -34.03  &  -2.32  &  92.53  & -49.82  &  8.70  & -53.98 \\
   21.68  &  19.69  & -21.72  &  28.90  &  6.50  & -16.12 \\ 
   30.69  & -20.02  &  13.80  &  4.41  &  52.42  & -30.91 \\ 
   13.63  &  18.86  & -12.82   &  3.87 &  -11.02  &  88.71 \end{pmatrix}. $$

Though this matrix appeared to be diagonally stable, it is neither diagonally dominant (by rows or columns) nor sign-symmetric (since the product $a_{ij}a_{ji}$ varies its sign).

\subsection{Parameter-dependent analysis}
 Consider the following $4 \times 4$ parameter-dependent matrix:
$${\mathbf A} = \begin{pmatrix} 1 & 0 & p_1 & p_2 \\
1 & 1 & 0 & 0 \\
1 & 1 & 1& 0  \\ 
1& 1 & 1& 1  \end{pmatrix},$$
where $p_1, \ p_2 $ are the real parameters.

The characteristic polynomial of the $4 \times 4$ matrix $\mathbf{A}_4$ is:
$$f_{\mathbf A}(\lambda) = \lambda^4 - 4\lambda^3 + (6 - p_1 - p_2)\lambda^2 + (p_1 - 4)\lambda + 1$$ 
Hence $$f_{\mathbf A}(-\lambda) = \lambda^4 + 4\lambda^3 + (6 - p_1 - p_2)\lambda^2 + (4 - p_1)\lambda + 1.$$
By considering its Hurwitz matrix
$$H(f_{\mathbf A}(-\lambda)) = \begin{pmatrix} 4 & 4 - p_1 & 0 & 0 \\
1 & 6 - p_1 - p_2 & 1 & 0 \\
0 & 4 & 4 - p_1 & 0  \\
0 & 1 & 6 - p_1 - p_2 & 1 \\\end{pmatrix}$$

By solving the system of inequalities for the principal minors of the Routh--Hurwitz matrix
$$ \begin{cases}
    \Delta_1 = 4 > 0;\\
\Delta_2 = 20 - 3p_1 - 4p_2 > 0; \\
\Delta_3 =\Delta_4 = 3p_1^2 + 4p_1 p_2 - 32p_1 - 16p_2 + 64 > 0, \\
\end{cases}$$
we find the set of $(p_1, p_2)$ for which $\mathbf A$ is stable:
$$ \{p_1 < 4\} \cap \{p_2 < 5 - \frac{3p_1}{4} + \frac{4}{p_1 - 4} \}. $$

To maximize the diagnostic capability of our framework, we use elementary properties of $D$-stable matrices (see Lemma \ref{el}). Let $\Sigma[4]$ all the possible permutations of the set $\{1, \ 2, \ 3, \ 4\}$, and let $\mathbf{P}_{\sigma}$ represent the orthogonal permutation matrix corresponding to a coordinate mapping $\sigma \in \Sigma[4]$. Lemma \ref{el} (p. 3) implies that {\it if any isolated permutation $$\mathbf{A}_{\sigma} := \mathbf{P}_{\sigma}\mathbf{A}\mathbf{P}^T_{\sigma}$$ satisfies the conditions of Test I or II, then the initial matrix $\mathbf{A}$ is guaranteed to be $D$-stable}.

When evaluating the parameterized family $\mathbf{A}(p_1, p_2)$ under our recursive framework, different permutation matrices $\mathbf{P}_{\sigma}$ yield structurally distinct coefficient tracking paths. Consequently, applying Test I at the fixed terminal level ($i$) to all $4! = 24$ individual matrix permutations results in several different certified $D$-stability regions. To extract the largest possible region of guaranteed $D$-stability, we define the expanded certified variety as the set-theoretic union of these independent parametric zones:
$$\Omega = \bigcup_{\sigma \in \Sigma[4]} \left\{ (p_1, p_2) \in \mathbb{R}^2 : \mathbf{A}_{\sigma}(p_1, p_2) \text{ satisfies Test I}\right\}.$$ 
This structural coordinate recycling reduces the algebraic conservatism of the terminal leaf relaxations. To validate the efficiency of this permutation-aided expansion and evaluate its conservatism relative to exact bounds, we compare the obtained parameter regions to the OliveiraPeres LMI test \cite{OLP}. The reference boundary parameters for which $\mathbf A$ is $D$-stable were computed using a brute-force grid search via the Oliveira--Peres LMI test, evaluated over a fine mesh grid with a uniform discretization step size $\Delta p = 0.1$. 
As a result, we find matrices, for example, $${\mathbf A} = \begin{pmatrix} 1 & 0 & 1 & -10 \\
1 & 1 & 0 & 0 \\
1 & 1 & 1& 0  \\ 
1& 1 & 1& 1  \end{pmatrix} $$
that are recognized by Test I but not recognized by LMI test. Note, that this matrix is $\bf not$ diagonally stable (in general, Oliveira--Peres LMI test is likely to fail on the boundary of the set of $D$-stable matrices, giving false-negative responses).

Now, consider the following parameter-dependent matrix family, defined by applying Test I at $i=3$.
\begin{equation}\label{Form} {\mathbf A}(p):= \begin{pmatrix} 1 & 0 & p & 2p \\
1 & 1 & 0 & 0 \\
1 & 1 & 1& 0  \\ 
1& 1 & 1& 1  \end{pmatrix}\end{equation}

To summarize, we highlight that matrices of Form $\eqref{Form},$ though successfully recognized as $D$-stable by Test I even at its most conservative terminal level ($i = 3$), fail to satisfy any of the classical sufficient conditions cataloged in the classification by Johnson $\cite{JOHN1}.$ Specifically, for $p \leq -8$, the parameterized matrix $\mathbf{A}(p)$ exhibits the following structural properties outside of traditional classes:

\begin{enumerate} \item[\rm 1.] {\bf Diagonal Incompatibility:} The matrix is not diagonally stable for $p \le -8$. To establish the boundaries of the parameter $p$ for which $\mathbf{A}(p)$ admits the property of diagonal stability, we consider its leading principal submatrix 
$${\mathbf A}|_4 = \begin{pmatrix} 1 & 0 & p \\
1 & 1 & 0  \\
1 & 1 & 1  \\ 
1& 1 & 1 \end{pmatrix} $$ and apply the diagonal stability criterion for $n=3$ (see \cite{CROSS}, also \cite{KAB}, p. 76, Fact 2.8.2). Hence we obtain that ${\mathbf A}|_4$ is diagonally stable if and only if and only if $-8 < p < 1$ providing that $\mathbf A$ is {\bf not} diagonally stable for all $p \leq -8$ and $p \geq 1$. 
\item[\rm 2.] {\bf Non-$M$-Matrix Structure:} The matrix cannot be classified as an $M$-matrix due to the immediate violation of sign-symmetry given by $a_{13}a_{31} = p < 0$ (for negative values of $p$, the product of symmetric off-diagonal entries fails to be even nonnegative).
\item[\rm 3.] {\bf Absence of Diagonal Dominance:} The matrix is structurally non-diagonally dominant with respect to both row and column sums across all $p$ such that $| p| > \frac{1}{2}$.
\item[\rm 4.] {\bf Non-Triangularity:} The configuration contains non-zero entries in both the strict upper and lower triangles ($a_{31}=1$, $a_{13}=p$). 
\item[\rm 5.] {\bf Absence of Sign-Stability:} The matrix fails the sign-stability criteria due to a lack of complete qualitative sign-symmetry across its off-diagonal pairings.
\item[\rm 6.] { \bf Non-Tridiagonal Banding:} The matrix contains non-zero peripheral entries ($a_{14} = 2p$, $a_{41} = 1$) that violate tridiagonal patterns.
\item[\rm 7.] {\bf Non-Oscillatory:} For $p<0$, the matrix $\mathbf{A}(p)$ is not even positive.
\item[\rm 8.] {\bf Failure of Condition 8:} The matrix violates Condition 8 of \cite{JOHN1} (p. 58) as a direct consequence of losing its diagonal stability outside the interval $(-8, 1)$.
\item[\rm 9.] {\bf Vulnerability to Hadamard Instability:} For every parameter choice $p \ge 1$ and $p \le -8$, there exists a symmetric positive definite matrix $\mathbf{P}_0$ such that the Hadamard (entry-wise) product $\mathbf{P}_0 \circ \mathbf{A}(p)$ is unstable. This is also a direct consequence of losing its diagonal stability (see \cite{KRA}, p. 246, Theorem 1.2).
\item[\rm 10.] {\bf Lack of Sign-Symmetry:} The structural configuration fails sign-symmetry restrictions for all negative $p$ ($a_{13}a_{31} = p < 0$).
\item[\rm 11--12.] {\bf Dimensional Exogenousness:} The dimension of the matrix ($n = 4$) places it strictly beyond the low-dimensional regimes ($n \le 3$) where simpler algebraic characterizations of $D$-stability become necessary and sufficient.
\item[\rm 13.] {\bf Failure of GKK Conditions:} The matrix does not satisfy the Gantmacher--Krein--Kotelyansky (GKK) property (see \cite{JOHN1},\cite{HS}), as it lacks even week sign-symmetry.
\end{enumerate}
Consequently, the parameterized matrix family $\eqref{Form}$ belongs to an independent algebraic variety of $D$-stable matrices that escapes all thirteen classical sufficient conditions.

\section*{Conclusions} The method, proposed in this paper, allows to transform the continuous parameter-dependent condition $\det({\mathbf A} + i{\mathbf D}) \neq 0$ into a discrete combinatorial decomposition expressed in terms of principal minors. However, several limitations of the proposed method should be acknowledged, such as:
\begin{enumerate}
    \item[-] exponential complexity in the worst case. Easy stopping can mitigate this, and the structured subclasses (e.g. matrices with zero diagonal entries, zero principal minors or symmetries) reduce complexity considerably.
    \item[-] the computation can be numerically challenging for ill-conditioned matrices.
\end{enumerate}

\section*{Acknowledgements} The research was supported by the State Research Program "Convergence-2030" of NAS of Belarus (Project No. 1.08.2).

The author expressed deep gratitude to K. Jilich for developing the C++ software implementation and executing the numerical experiments presented in Section 7.

During the preparation of this work, the author used Gemini to assist with proofreading and polishing English grammar. The author reviewed and edited the output as needed and takes full responsibility for the content of the published article.

\end{document}